\documentclass[11pt]{amsart}
\usepackage{mathptmx}
\usepackage{amssymb}
 \usepackage{xcolor,soul,framed,caption}
 \colorlet{shadecolor}{yellow}

\usepackage{amsmath}
\usepackage{amssymb,mathtools}
\usepackage{stmaryrd}
\usepackage{hyperref}
\usepackage{graphicx}
\usepackage{color}
\usepackage{bbm}
\usepackage{epstopdf}

\def\XXint#1#2#3{{\setbox0=\hbox{$#1{#2#3}{\int}$} \vcenter{\hbox{$#2#3$}}\kern-.5\wd0}}

{   \end{list} }

\newcommand{\R}{\mathbb{R}}
\newcommand{\te}{\textrm}
\newcommand{\omo}{\Omega \!- \!\Omega}
\newcommand{\omok}{\Omega_{k}\!-\!\Omega_{k}}
\newcommand{\tacka}{\,\cdot\,}
\newcommand{\veps}{\varepsilon}

\newcommand{\lamw}{\lambda_{W}}
\newcommand{\lamv}{\lambda_{V}}

\DeclareMathOperator{\supp}{supp}

\DeclareMathOperator{\dist}{dist} \DeclareMathOperator{\dive}{div}
\DeclareMathOperator{\Hess}{Hess} \DeclareMathOperator{\proj}{proj}
\DeclareMathOperator{\Id}{Id} \DeclareMathOperator{\Conv}{Conv}
\DeclareMathOperator{\Diam}{diam}

\definecolor{mygreen}{rgb}{0.1,0.75,0.2}

\newcounter{broj}

\newtheorem{thm}{Theorem}[section]
\newtheorem{lem}[thm]{Lemma}

\newtheorem{prop}[thm]{Proposition}
\theoremstyle{definition}
\newtheorem{defi}[thm]{Definition}
\newtheorem{remk}[thm]{Remark}

\numberwithin{equation}{section}

\title{nonlocal-interaction equations on Uniformly Prox-regular Sets }
\author{Jos\'{e} A. Carrillo}
\address{Department of Mathematics, Imperial College London, SW7 2AZ London, United Kingdom}
\email{carrillo@imperial.ac.uk}

\author{Dejan Slep\v{c}ev}
\address{Department of Mathematical Sciences, Carnegie Mellon
University, Pittsburgh, PA 15213 USA}
\email{slepcev@math.cmu.edu}

\author{Lijiang Wu}
\address{Department of Mathematical Sciences, Carnegie Mellon
University, Pittsburgh, PA 15213 USA}
\email{lijiangw@andrew.cmu.edu}

\begin{document}

\begin{abstract}
We study the well-posedness of a class of nonlocal-interaction equations on general domains $\Omega\subset \mathbb{R}^{d}$, including nonconvex ones. We show that under mild assumptions on the regularity of domains (uniform prox-regularity), for  $\lambda$-geodesically convex  interaction and external potentials, the nonlocal-interaction equations
have unique weak measure solutions. Moreover, we show quantitative estimates on the stability of solutions which quantify the interplay of the geometry of the domain and the convexity of the energy.
We  use these results to investigate on which domains and for which potentials the solutions aggregate to a single point as time goes to infinity.
Our approach is based on the theory of gradient flows in spaces of probability measures.
\end{abstract}
\date{\today}

\maketitle

\section{Introduction}\label{sec:intro}

\subsection{Description of the problem.}
We study a continuum model of agents interacting via a potential $W$ and subject to an external potential $V$ confined to a closed subset $\Omega\subset\mathbb{R}^{d}$. 
Such systems arise in modeling macroscopic behavior of agents interacting in geometrically confined domains. The domain boundary may be an environmental obstacle, like a river, or the ground itself, as in the models of locust patterns discussed in \cite{BT,TopBerLogToo,TDEB}.
We consider systems in which the environmental boundaries limit the movement but not the interaction between agents. To illustrate, the the agents can still see each others over a river even if they are not able to traverse it.

We describe configurations of agents as measures supported on the given domain. This allows to study both the discrete case, when an individual agent carries a positive mass, and the continuum limit in which a system with many agents is described by a function giving the density of agents. 
The measure describing the agents interacting over time satisfies a nonlocal-interaction equation in the sense of weak measure solutions.
The theory of weak measure solutions to nonlocal-interaction equations was developed in \cite{AGS, CDFLS}.
In \cite{WS1} systems of interacting agents on domains with boundary were considered in
a setting which allowed for heterogeneous environments, but required the sets to be convex with $C^1$ boundary. Here we consider general domains which are not required to be convex and whose boundary may not be differentiable.
 The geometrical confinement introduces a constraint on the possible velocity fields 
of the agents at the boundary.
We consider the situation in which there is no additional friction at the boundary.
More precisely, for smooth domains, the velocity of the agents at the boundary is the projection of what the velocity would be, for the given configuration if there was no boundary, to the half plane of vectors pointing inside the domain. That is, inward pointing velocities at the boundary are unchanged, while the outward pointing velocities are projected on the tangent plane  to the boundary.
The measure $\mu(\tacka)$ describing the agent configuration becomes a distributional solution of the equation
\begin{equation}\label{conequ}
\left\{
\begin{aligned}
& \frac{\partial}{\partial
t}\mu(t,x)-\dive\left(\mu(t,x)P_x\left(-\int_{\Omega}\nabla
W(x-y)d\mu(t,y)-\nabla
V(x)\right)\right)=0,\\
& \mu(0)=\mu_{0},\\
\end{aligned}
\right.
\end{equation}
where $P_x$ is the projection of the velocities to inward pointing ones.

When considering domains which are not $C^1$ the question is what should the velocity of agents be 
at a boundary point where the domain is not differentiable.
 Similar questions have been encountered in studies of differential inclusions on moving domains (general sweeping processes), see  \cite{ET1,ET2,Ven} and references therein. We rely on notions developed there to properly define the cone of admissible directions at a boundary points and the proper way to project the velocity to the allowable cone. In particular we consider the equation \eqref{conequ} with projection $P_x$ defined in  \eqref{proj} ($P_x = P_{T(\Omega,x)}$). 

While one would like to consider very general domains there are limits to possible domains on which a well-posedness of measure solutions can be developed. Namely, if the 
domains have an inside corner, then it is not possible for the measure solutions of \eqref{conequ} to be stable, as we discuss in Remark \ref{rem:instab}. It turned out that a class of domains which is rather general and allows for a well-posedness theory are the (uniformly) prox-regular domains (see Definition \ref{def:proxset}).
Prox-regular domains  are  the sets which have an outside neighborhood 
such that for each of its points there exists a unique closest point on the boundary. 
In particular prox-regular domains can have outside corners and outside cusps, but not inside corners.


Our main result is the well-posedness of weak measure solutions, described in Definition \ref{def:wms}, of the nonlocal-interaction equation \eqref{conequ} on uniformly prox-regular domains. To show it we rely on further structure the equation possesses. Namely to the interaction $W$, we associate
interaction energy
$$\mathcal{W}(\mu)=\frac{1}{2}\int_{\Omega}\int_{\Omega}W(x-y)d\mu(x)d\mu(y),$$
the and to potential $V$, the  potential energy
$$\mathcal{V}(\mu)=\int_{\Omega}V(x)d\mu(x).$$ 
We define the energy $\mathcal{E}$ by
\begin{equation}\label{def:totaene}
\mathcal{E}(\mu)=\mathcal{W}(\mu)+\mathcal{V}(\mu).
\end{equation}
The energy $\mathcal{E}$ is a dissipated quantity of the evolution \eqref{conequ}, and furthermore the equation can be interpreted as the gradient flow of the energy with respect to the Wasserstein metric. 
Our strategy is to first show the existence of the gradient flow solutions to (\ref{conequ}) in the space of probability measures endowed with the Wasserstein metric. The gradient flow in the space of probability measures endowed with Wasserstein metric was first used in \cite{JKO} for Fokker-Planck equations.  The gradient-flow approach to well-posedness of nonlocal-interaction equations was developed in \cite{AGS, CDFLS} and extended to $C^1$ domains with boundary in \cite{WS1}. Furthermore the gradient flow approach was used to study  systems in which 
there are state constraints that determine the set of possible velocities, in particular in 
crowd motion models \cite{AleKimYao,MRS,MRSV} where the  constraint on the $L^{\infty}$-norm of the density of agents, which leads to an $L^{2}$-projection of velocity field.

After establishing the well-posedness of gradient-flow solutions we show the well-posedness of weak measure solutions. 

To show the existence of gradient flow solutions, we use particle approximations, that is we use a sequence of delta masses
$\mu^{n}_{0}=\sum_{j=1}^{k(n)} m_{j} \delta_{x^{n}_{j}}$ to approximate the initial data $\mu_0$ and solve (\ref{conequ}) with initial data $\mu^{n}_{0}$. Here the notion of gradient flow solutions (and weak measure solutions) provides the advantage that we can work with delta measures, which makes the particle approximation meaningful. With discrete initial data $\mu^{n}_0$, (\ref{conequ}) becomes a system of ordinary differential equations, we solve the ODE system and
prove that the solutions $\mu^{n}(\tacka)$ converges to some $\mu(\tacka)$ by establishing the stability property of solutions to (\ref{conequ}) with different initial data. We then show that the limit curve $\mu(\tacka)$ is a gradient flow solution to (\ref{conequ}) with initial data $\mu_0$ by proving that $\mu(\tacka)$ achieves the maximal dissipation of the associated energy,  and is thus the steepest descent of the energy.

The novelty here is that even though the domain $\Omega$ is only prox-regular (not necessarily convex or $C^{1}$) and the velocity field is discontinuous (due to the projection $P$), the ODE systems are still well-posed (refer to Theorem \ref{existode}) and the stability of solutions $\mu^{n}(\tacka)$ in Wasserstein metric $d_{W}$ is valid with explicit dependence on the prox-regularity constant (refer to Proposition \ref{contrprop}). Under semi-convexity assumptions on the potential functions $W$ and $V$, this enables us to show the well-posedness, that is existence and stability of weak measure solutions to (\ref{conequ}) in three different cases: $\Omega$ bounded and prox-regular (Theorem \ref{exgrad} and Thorem \ref{stagrad}), $\Omega$ unbounded and convex (Theorem \ref{globexis}), and $\Omega$ unbounded and prox-regular with compactly supported initial data $\mu_0$ (Theorem \ref{locexistence}). We can also generalize the well-posedness results to time-dependent interaction and external potentials
$W=W(t,x), V=V(t,x)$ (Remark \ref{timedep}). We also give sufficient conditions on the shape of $\Omega$ to ensure the existence of an interaction potentials $W$ such that 
solutions $\mu(\tacka)$ to (\ref{conequ}) aggregate to a single delta mass as time goes to infinity (Theorem \ref{aggrecon} and Remark \ref{sharpness}) .

\subsection{Description of weak measure solutions}
Let $\mathcal{P}(\Omega)$ be the  space of probability
measures on $\Omega$ and let
$$
\mathcal{P}_{2}(\Omega)=\left \{\mu\in \mathcal{P}(\Omega):
\int_{\Omega}|x|^{2}d\mu(x)< \infty\right\}\,,
$$
the space of probability measures with finite second moment. $\mathcal{P}_{2}(\Omega)$
is a complete metric space endowed with the 2-Wasserstein metric
\begin{equation}\label{def:Wassmetric}
d_{W}^{2}(\nu,\mu)= \min\left\{\int_{\Omega\times \Omega}
|x-y|^{2}d\gamma(x,y): \gamma\in \Gamma(\mu,\nu)\right\}\,,
\end{equation}
where $\Gamma(\mu,\nu)$ is the set of transportation plans between $\nu$ and $\mu$, that is the set of joint probability distributions on $\Omega \times \Omega$ with first marginal $\mu$
and second marginal $\nu$:
$$
\Gamma(\mu,\nu)=\left\{\gamma\in\mathcal{P}(\Omega\times\Omega):
(\pi_{1})_{\sharp}\gamma=\mu,(\pi_{2})_{\sharp}\gamma=\nu\right\}.
$$
We refer to books \cite{V1,V2} for theory of optimal transport.
We denote the set of optimal transport plans between $\mu$ and $\nu$ (with respect to 2-Wasserstein metric) by
$\Gamma_{o}(\mu,\nu)$, that is
\begin{equation}\label{def:Optplan}
\Gamma_{o}(\mu,\nu)=\left\{\gamma\in\Gamma(\mu,\nu):\int_{\Omega\times\Omega}|x-y|^{2}d\gamma(x,y)=d_{W}^{2}(\mu,\nu)\right\}.
\end{equation}

We now give the definition of weak measure solutions to the continuity equation.
\begin{defi} \label{def:wms}
A locally absolutely continuous curve $\mu(\tacka)\in
\mathcal{P}_{2}(\Omega)$ is a \emph{weak measure solution} to
(\ref{conequ}) with initial value $\mu_{0}$ if
$$
v(t,x)=-P_{x} \left( \int_{\Omega}\nabla W(x-y)d\mu(y)+\nabla V(x)\right)\in
L^{1}_{loc}([0,+\infty);L^{2}(\mu(t)))$$ and
\begin{align*}
\int_{0}^{\infty}\!\!\!\!\int_{\Omega} &\frac{\partial \phi}{\partial
t}(t,x)d\mu(t,x)dt+\int_{\Omega}\phi(0,x)d\mu_{0}(x)
+ \int_{0}^{\infty}\!\!\!\!\int_{\Omega}\left\langle\nabla \phi(t,x),v(t,x)\right\rangle d\mu(x)=0\,,
\end{align*}
for all $\phi\in C_{c}^{\infty}([0,\infty)\times \Omega)$. The projection $P_x$ is described below and
formally defined in \eqref{proj} (with $P_x = P_{T(\Omega,x)}$).
\end{defi}
Note that the test function $\phi$ does not have to be zero on the
boundary of $\Omega$, and thus the no-flux boundary condition is
imposed in a weak form.

We now define the projection $P_x$. When $\partial\Omega\in C^1$ is smooth
and oriented, the definition of $P_x$ is given in \cite{WS1, WS2}, and it is  given by
$P_x (v)=v-\langle v,\nu(x)\rangle\nu(x)$ if $\langle v,\nu(x)\rangle>0$ and $P_x (v)=v$ otherwise, where $\nu(x)$ is the unit outward normal vector to the boundary at $x\in\partial \Omega$.
When $\Omega$ is only prox-regular, to define $P_x$, we need to recall some notations
from non-smooth analysis, see \cite{Bou,CSW}, in order to replace the normal vector field and the inward and outward directions.
\begin{defi}
Let $S$ be a closed subset of $\mathbb{R}^{d}$. We define the
\emph{proximal normal cone} to $S$ at $x$ by,
\begin{equation*}
N^{P}(S,x)=\left\{v\in\mathbb{R}^{d}\, :\, \exists \alpha>0,\; x\in
P_{S}\left(x+\alpha v\right)\right\},
\end{equation*}
where
\begin{equation*}
P_{S}(y)=\left\{z\in S \, : \,\inf_{w\in S}|w-y|=|z-y|\right\}
\end{equation*}
is the projection of $y$ onto $S$.
\end{defi}
Note that for $x\in S\setminus\partial S, N^{P}(S,x)=\{0\}$ and by
convention for $x\not\in S, N^{P}(S,x)=\emptyset$. The notion of
normal cone extends the concept of outer normal of a smooth set in
the sense that if $S$ is a closed subset of $\mathbb{R}^{d}$ with
boundary $\partial S$ an oriented $C^2$ hypersurface, then for each
$x\in \partial S$, $N^{P}(S,x)=\mathbb{R}^{+}\nu(x)$ where $\nu(x)$
is the unit outward normal to $S$ at $x$. We now recall the notion of
uniform prox-regular sets.
\begin{defi}\label{def:proxset}
Let $S$ be a closed subset of $\mathbb{R}^{d}$. $S$ is said to be
\emph{$\eta$-prox-regular} if for all $x\in\partial S$ and $v\in
N^{P}(S,x), |v|=1$ we have
\begin{equation*}
B_{\eta}(x+\eta v)\cap S=\emptyset,
\end{equation*}
where $B_{\eta}(y)$ denotes the open ball centered at $y$ with radius
$\eta>0$.
\end{defi}
Note that an equivalent characterization, see \cite{CSW,PolRocThi}, is given by: $S$ is
$\eta$-prox-regular if for any $y\in S,x\in\partial S$ and $v\in
N^{P}(S,x)$,
\begin{equation}\label{proxequiv}
\left\langle v,y-x\right\rangle\leq \frac{|v|}{2\eta}|y-x|^{2}.
\end{equation}
Observe that if $S$ is closed and convex, then $S$ is
$\infty$-prox-regular. We now turn to the tangent cones.
\begin{defi} \label{def:Ccones}
Let $S$ be a closed subset of $\mathbb{R}^{d}$ and $x\in S$, define
the \emph{Clarke tangent cone} by
\begin{equation*}
T^{C}(S,x)=\left\{v\in\mathbb{R}^{d}\,:\, \forall t_{n}\searrow 0,\forall
 x_{n} \in S, \te{ s.t. }  x_n \to x, \; \exists v_{n}\to v \te{ s.t. } (\forall n) \, x_{n}+t_{n}v_{n}\in
S \,\right\}\,,
\end{equation*}
and denote the \emph{Clarke normal cone} by
\begin{equation*}
N^{C}(S,x)=\left\{\xi\in \mathbb{R}^{n}: \langle \xi,v\rangle \leq 0
\,\forall v\in T^{C}(S,x)\right\}.
\end{equation*}
\end{defi}

\begin{figure}[h]
\centering
\includegraphics[width=0.45\linewidth,angle=270]{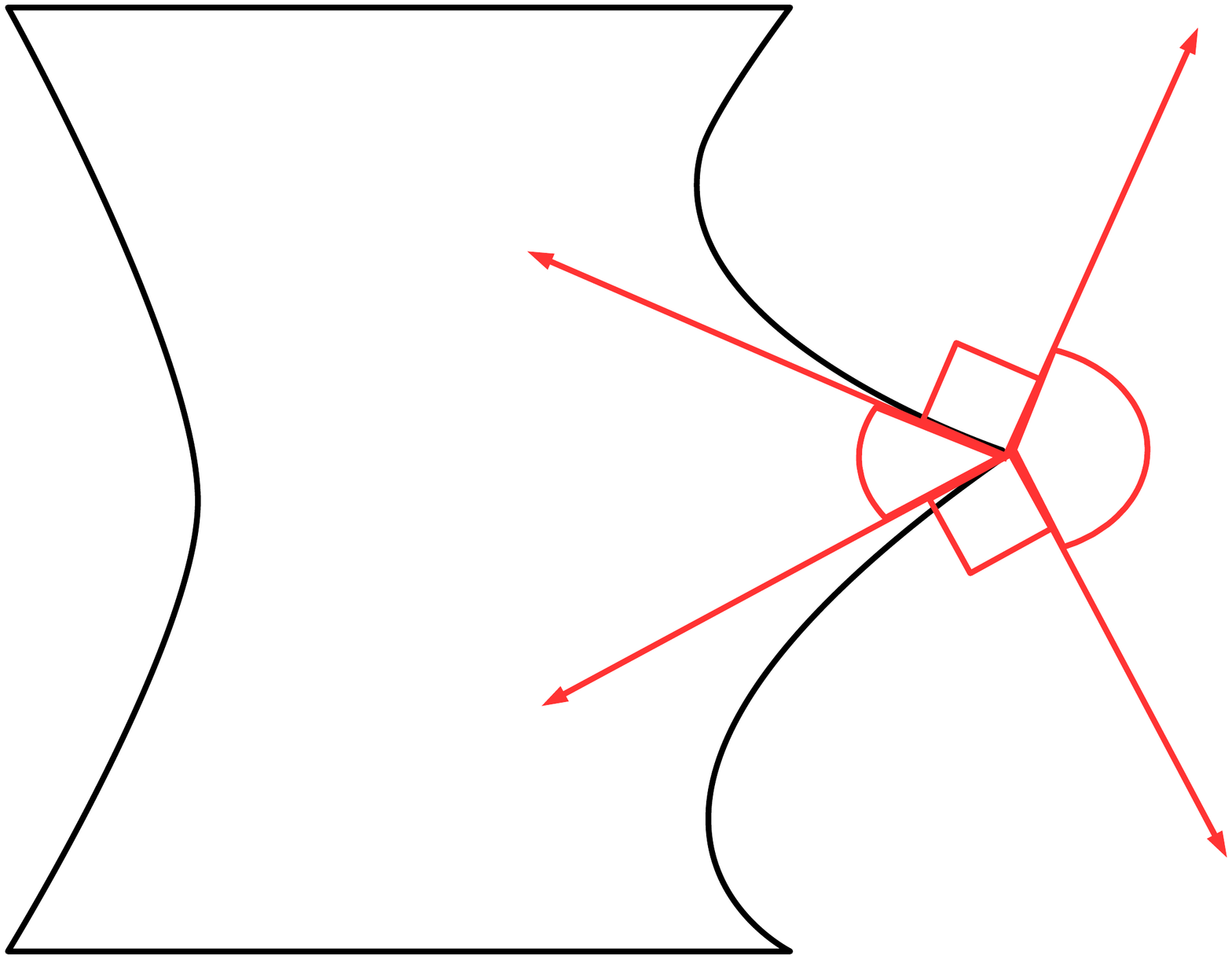}
\put(-200,-50){\Large $S$}
\put(-140,-105){$T(S,x)$}
\put(-55,-105){$N(S,x)$}
\put(-85,-130){$x$}

\caption{The set $S$ is prox-regular but not convex. At the corner point $x\in\partial S$, the tangent and normal cones are denoted by $T(S,x)$ and $N(S,x)$.}
\label{fig:cones}
\end{figure}

Note that $T^{C}(S,x), N^{C}(S,x)$ are closed convex cones, also by
convention $N^{C}(S,x)=\emptyset$ for all $x\not\in S$. In general,
we only have $N^{P}(S,x)\subset N^{C}(S,x)$ and the inclusion can be
strict. However, for $\eta$-prox-regular set $S$, we have
$N^{P}(S,x)=N^{C}(S,x)$, see \cite{CSW,PolRocThi}. In that case, we put
the normal cone and tangent cone as $N(S,x)$ and $T(S,x)$
respectively, and for any vector $w\in\mathbb{R}^{d}$, we define the
projection onto the tangent cone by $P_{T(S,x)}(w)$, i.e.,
\begin{equation} \label{proj}
P_{T(S,x)}\left(w\right)=\left\{v\in T(S,x): |v-w|=\inf_{\xi\in
T(S,x)}|\xi-w|\right\}.
\end{equation}
Since $T(S,x)$ is a closed convex cone, the infimum is
always attained, and $P_{T(S,x)}$ is well-defined. For notation
simplicity, since the set we are considering $\Omega$ is not
changing, we write $P_{x}$ instead of $P_{T(\Omega,x)}$ and
when the context is clear, we put $P$ for $P_{x}$. With these preliminaries, we can now state the main results of this work.

\subsection{Main results.}
For any set $A \subset \mathbb{R}^{d}$, we denote by
$A \! - \! A = \{x-y\,: x,y\in A\},$ and
the convex hull of $A$ by
$\Conv\left(A \right)=\{\theta x+(1-\theta)y: x,y\in A, 0\leq \theta\leq
1\}.$
For a function $f\in C^{1}(\mathbb{R}^{d})$, we say that $f$ is
$\lambda$-geodesically convex on a convex set $S$ if for any $x,y\in S$
we have
\begin{equation*}
f(y)\geq f(x)+\langle \nabla f(x),y-x\rangle
+\frac{\lambda}{2}|y-x|^{2}.
\end{equation*}
We call $f$ \emph{locally $\lambda$-geodesically convex} if there exist a
sequence of compact convex sets $K_{n}\subset\mathbb{R}^{d}$ and a
sequence of constants $\lambda_{n}$ such that $K_{n}\subset
K_{n+1}$, $\bigcup_{n} K_{n}=\mathbb{R}^{d}$ and $f$ is
$\lambda_{n}$-geodesically convex on $K_{n}$. Note that $f$ is
$\lambda$-geodesically convex on a convex set $S$ implies for any
$x,y\in S$
\begin{equation*}
\langle \nabla f(x)-\nabla f(y),x-y\rangle \geq \lambda|x-y|^{2}.
\end{equation*}

The main assumptions depend on the domain $\Omega$ and the support of initial data. In fact, we separate our results in three cases: 
$\Omega$ bounded, $\Omega$ unbounded and convex, and $\Omega$ unbounded with compactly supported initial data. The assumptions are very similar 
in nature based on the convexity of the potentials $V$ and $W$ and on their growth behavior at $\infty$ in the unbounded cases. 
We assume that both potentials $V$ and $W$ are $\lambda_V$- and $\lambda_W$-convex respectively, possibly locally convex. 
Finally, in case $V$ and $W$ are $\lambda$-locally convex, we can assume, without loss of generality, that $V$ and $W$ share the same sequence of 
compact convex sets, $K_{k}$ in the definition of locally $\lambda$-geodesic convexity, i.e., $K_{k}\subset K_{k+1}$, $\bigcup_{k\in\mathbb{N}}K_{k}=\mathbb{R}^{d}$ with $V$ 
and $W$ being $\lambda_{V,k}$ and $\lambda_{W,k}$-geodesically convex on $K_{k}$.

In case $\Omega$ is bounded, we assume that
\begin{itemize}
\item[\textbf{(M1)}] $\Omega\subset\mathbb{R}^{d}$ is $\eta$-prox-regular with $\eta>0$.
\item[\textbf{(A1)}] $W\in C^{1}(\mathbb{R}^{d})$ is $\lambda_{W}$-geodesically convex on
$\Conv\left(\omo\right)$ for some $\lambda_{W}\in\mathbb{R}$.
\item[\textbf{(A2)}] $V\in C^{1}(\mathbb{R}^{d})$ is
$\lambda_{V}$-geodesically convex on $\Conv\left(\Omega\right)$ for some $\lambda_{V}\in\mathbb{R}$.
\end{itemize}

We call a locally absolutely continuous curve
$\mu(t)\in\mathcal{P}_{2}(\Omega)$ a gradient flow of
the energy functional $\mathcal{E}$ defined in \eqref{def:totaene} if for a.e $t>0$
\begin{equation*}
v(t)\in -\partial\mathcal{E}(\mu(t)),
\end{equation*}
where $\partial\mathcal{E}(\mu(t))$ is the subdifferential of
$\mathcal{E}$ at $\mu(t)$ (as given in Definition \ref{subdif}) and $v(t)$ is the tangent velocity of the
curve $[0,\infty)\ni t\mapsto \mu(t)\in \mathcal{P}_{2}(\Omega)$ at
$\mu(t)$, which we recall in Section \ref{sec:bddomain}.

For a locally absolutely continuous curve $[0,T]\ni t \mapsto
\mu(t)\in \mathcal{P}_{2}(\Omega)$ with respect to 2-Wasserstein metric
$d_{W}$, we denote its metric derivative by
\begin{equation}\label{def:metrderi}
|\mu'|(t)=\limsup_{s\to t}\frac{d_{W}(\mu(t),\mu(s))}{|s-t|}.
\end{equation}

The main results of this paper is the well-posedness of weak measure solutions: existence and stability, with arbitrary initial data.
We establish it using an approximation scheme and the theory of gradient flows in spaces of probability measures.
\begin{thm}\label{exgrad}
Assume $\Omega$ is bounded and satisfies (M1) and $W, V$ satisfy
(A1), (A2). Then there exists a locally absolutely continuous curve
$\mu(\tacka)\in \mathcal{P}_{2}(\Omega)$ such that $\mu(\tacka)$ is a gradient
flow with respect to $\mathcal{E}$. Moreover, $\mu(\tacka)$ is a weak
measure solution to \eqref{conequ}.

 Furthermore for a.e.  $t>0$
\begin{equation}\label{metslo}
|\mu'|^{2}(t)=\int_{\Omega} \left|P_{x}\left(-\nabla
 W*\mu(r)(x)-\nabla V(x)\right)\right|^{2} d\mu(t,x),
\end{equation}
and for any $0\leq s\leq t<\infty$
\begin{equation}\label{enedis}
\mathcal{E}(\mu(s))=
\mathcal{E}(\mu(t))+\int_{s}^{t}\int_{\Omega}\left|P_{x}\left(-\nabla
 W*\mu(r)(x)-\nabla V(x)\right)\right|^{2}d\mu(r,x)dr.
\end{equation}
\end{thm}

\begin{thm}\label{stagrad}
Assume $\Omega$ is bounded and satisfies (M1) and $W, V$ satisfy
(A1), (A2). Let $\mu^{1}(\tacka),\mu^{2}(\tacka)$ be two weak measure
solutions to \eqref{conequ} with initial data
$\mu^{1}_{0},\mu^{2}_{0}$ respectively. Then
\begin{equation}\label{contra3}
d_{W}\left(\mu^{1}(t),\mu^{2}(t)\right)\leq
\exp\left(\left(-\lamw^{-}-\lamv +\frac{\|\nabla
W\|_{L^{\infty}(\omo)}+\|\nabla
V\|_{L^{\infty}(\Omega)}}{\eta}\right)t\right)
d_{W}\left(\mu^{1}_{0},\mu^{2}_{0}\right).
\end{equation}
for any $t\geq 0$ where $\lamw^- = \min \{\lamw,0\}$. Moreover, the weak measure solution is
characterized by the system of Evolution Variational Inequalities:
\begin{equation} \label{EVI3}
\frac{1}{2}\frac{d}{dt}d_{W}^{2}\left(\mu(t),\nu\right)+
\left(\frac{\lamw^-}{2}+\frac{\lamv}{2}-\frac{\|\nabla W\|_{L^{\infty}(\omo)}+\|\nabla
V\|_{L^{\infty}(\Omega)}}{2\eta}\right)d_{W}^{2}\left(\mu(t),\nu\right)\leq
\mathcal{E}(\nu)-\mathcal{E}(\mu(t)),
\end{equation}
for a.e. $t>0$ and for all $\nu\in\mathcal{P}_{2}(\Omega)$.
\end{thm}

Observe that in the stability estimate for solutions \eqref{contra3}, we find two contributions due to the $\lambda$-convexity of the potentials and the $\eta$-prox-regular property of the domain $\Omega$ respectively.

On $\R^n$ when $\mu^1(0)$ and $\mu^2(0)$ have the same center of mass $\lambda_W^-$ can be replaced by $\lambda_W$ in \eqref{contra3}. Thus when the potential $W$ is uniformly geodesically convex, $\lambda_W>0$ and thus there is exponential contraction of solutions. On bounded domains this is not the case since interaction with boundary can change the center of mass of a solution. Nevertheless part of the claim can be recovered. We consider the case that $V \equiv 0$. Denote  the set of singletons by $\Xi=\{\delta_{x}:x\in\mathbb{R}^d\}$. Note that we included the singletons which are not in the set $\Omega$, since the center of mass for measures on a non-convex $\Omega$ may lie outside the set.

\begin{prop} \label{single_contr}
Assume $\Omega$ is bounded and satisfies (M1) and $W$ satisfies
(A1). Let $\mu(\tacka)$ be a weak measure
solutions to \eqref{conequ} with $V\equiv 0$. Then
\begin{equation}\label{contra3}
d_{W}\left(\mu(t),\Xi\right)\leq
\exp\left(\left(-\lamw+\frac{\|\nabla
W\|_{L^{\infty}(\omo)}}{\eta}\right)t\right)
d_{W}\left(\mu_{0},\Xi \right).
\end{equation}
for any $t\geq 0$.
\end{prop}

The proposition implies that solution can aggregate to a point (in perhaps infinite time) even on a nonconvex domain. We ask on what domains there exists a potential for which for any initial datum this aggregation property holds. We provide a sufficient condition on the shape of $\Omega$ for aggregation to hold: Let $\Diam(\Omega)=\sup_{x,y\in\Omega}|x-y|$.

\begin{thm}\label{aggrecri}
 Assume that $\Omega$ is bounded and satisfies (M1). If $\eta>\frac{1}{2}\Diam(\Omega)$, then for external potential $V\equiv 0$, there exists an interaction potential $W$ satisfying (A1) for some $\lamw>0$, and constant $C(\Omega)<0$   such that 
 \begin{equation}\label{aggrecon}
  d_{W}(\mu(t),\Xi)\leq d_{W}(\mu_0,\Xi)\exp\left(C\left(\Omega\right)t\right),
 \end{equation}
for all $t \geq 0$. In particular, the solution aggregates to a singleton:
 \begin{equation}\label{aggrecon2}
 \lim_{t\to\infty}d_{W}\left(\mu(t),\Xi\right)= \lim_{t\to\infty}d_{W}\left(\mu(t),\delta_{\bar x(t)}\right)=0\,,
 \end{equation}
where $\bar x(t)=\int_{\Omega} x d\mu(t)$ is the center of mass for $\mu(t)$.
\end{thm}
Note that the constant in  $\eta>\frac{1}{2}\Diam(\Omega)$ cannot be improved, as the example in Remark \ref{sharpness} shows.
\medskip

We generalize the two existence and stability results to the unbounded case in two different settings.
In case $\Omega$ is unbounded, and for general initial data
$\mu_{0}$, possibly with noncompact support, we give
the global assumptions: for some constants $\lamw,\lamv\in\mathbb{R}$
and $C>0$,
\begin{itemize}
\item[\textbf{(GM1)}] $\Omega\subset\mathbb{R}^{d}$ is convex, i.e.,
$\Omega$ is $\infty$-prox-regular.
\item[\textbf{(GA1)}] $W\in C^{1}(\mathbb{R}^{d})$ is $\lamw$-geodesically convex on
$\Conv\left(\omo\right)=\omo$.
\item[\textbf{(GA2)}] $\nabla W$ has linear growth, i.e., $|\nabla
W(x)|\leq C(1+|x|)$ for all $x\in \mathbb{R}^{d}$.
\item[\textbf{(GA3)}] $V\in C^{1}(\mathbb{R}^{d})$ is
$\lamv$-geodesically convex on $\Conv\left(\Omega\right)=\Omega$.
\item[\textbf{(GA4)}] $\nabla V$ has linear growth, $|\nabla V(x)|\leq
C(1+|x|)$ for all $x\in\mathbb{R}^{d}$.
\end{itemize}
The main result in this setting reads as:

\begin{thm}\label{globexis}
Assume $\Omega$ is unbounded and satisfies (GM1) and $W, V$ satisfy
(GA1)-(GA4), then for any $\mu_{0}\in \mathcal{P}_{2}(\Omega)$,
there exists a gradient flow solution $\mu(\tacka)$ with respect to $\mathcal{E}$
such that $\mu(\tacka)$ is a weak measure solution to
\eqref{conequ}. Moreover, for a.e. $t>0$
\begin{equation*}
|\mu'|^{2}(t)=\int_{\Omega} \left|P_{x}\left(-\nabla
 W*\mu(r)(x)-\nabla V(x)\right)\right|^{2} d\mu(t,x),
\end{equation*}
and for any $0\leq s\leq t<\infty$
\begin{equation*}
\mathcal{E}(\mu(s))=
\mathcal{E}(\mu(t))+\int_{s}^{t}\int_{\Omega}\left|P_{x}\left(-\nabla
 W*\mu(r)(x)-\nabla V(x)\right)\right|^{2}d\mu(r,x)dr.
\end{equation*}
Similarly, if $\mu^{1}(\tacka),\mu^{2}(\tacka)$ are two weak measure solutions
to (\ref{conequ}) with initial data $\mu^{1}_{0},\mu^{2}_{0}$
respectively, then
\begin{equation}\label{globcontra}
d_{W}\left(\mu^{1}(t),\mu^{2}(t)\right)\leq \exp\left(-\left(\lamw^{-}+\lamv\right)
t\right) d_{W}\left(\mu^{1}_{0},\mu^{2}_{0}\right).
\end{equation}
for any $t\geq 0$. Also the weak measure solution is characterized
by the system of Evolution Variational Inequalities:
\begin{equation} \label{EVI4}
\frac{1}{2}\frac{d}{dt}d_{W}^{2}\left(\mu(t),\nu\right)+ \left(\frac{\lamw^{-}}{2}+\frac{\lamv}{2}\right)
d_{W}^{2}\left(\mu(t),\nu\right)\leq
\mathcal{E}(\nu)-\mathcal{E}(\mu(t)),
\end{equation}
for a.e. $t>0$ and for all $\nu\in\mathcal{P}_{2}(\Omega)$.
\end{thm}

Since $\Omega$ is convex means $\Omega$ is $\infty$-prox-regular, the stability estimate (\ref{globcontra}) and EVI (\ref{EVI4}) in the convex setting are consistent with the estimates in the $\eta$-prox-regular setting by taking $\eta=\infty$ in (\ref{contra3}) and (\ref{EVI3}).

The convexity assumption is needed since on nonconvex unbounded domains we do not know how to control the error due to lack of convexity (as measured by the prox-regularity \eqref{proxequiv}) in the stability of solutions. However, we can show that control assuming compactly supported initial data. Therefore, when $\Omega$ is unbounded and the initial data $\mu_{0}$ has compact support, we assume there exist some constants $\eta>0, \lamw,\lamv\in \mathbb{R},C>0$ such that the following local assumptions
hold
\begin{itemize}
\item[\textbf{(M1)}] $\Omega\subset\mathbb{R}^{d}$ is
$\eta$-prox-regular.
\item[\textbf{(LA1)}] $W\in C^{1}(\mathbb{R}^{d})$ is locally $\lambda$-geodesically convex on
$\mathbb{R}^{d}$.
\item[\textbf{(LA2)}] $\nabla W$ has linear growth, i.e., $|\nabla
W(x)|\leq C(1+|x|)$ for all $x\in \mathbb{R}^{d}$.
\item[\textbf{(LA3)}] $V\in C^{1}(\mathbb{R}^{d})$ is locally
$\lambda$-geodesically convex on $\mathbb{R}^{d}$.
\item[\textbf{(LA4)}] $\nabla V$ has linear growth, $|\nabla V(x)|\leq
C(1+|x|)$ for all $x\in\mathbb{R}^{d}$.
\end{itemize}
Note that the conditions (LA1) and (LA3) are satisfied whenever $V$ and $W$ are $C^2$ functions on $\R^d$, which is the case in many practical applications.

We show in this setting the following theorem about existence and stability for
weak measure solutions for initial data with compact support.
\begin{thm}\label{locexistence}
Given that $\Omega$ is unbounded and satisfies (M1), and $W, V$
satisfy (LA1)-(LA4). If $\supp(\mu_0)\subset\Omega$ is compact, say
$\supp(\mu_0)\subset B(r_0)\cap \Omega$, then there exists a weak
measure solution $\mu(\tacka)$ to \eqref{conequ} such that
$\supp(\mu(t))\subset B(r(t))$ for $r(t)=(r_{0}+1)\exp(Ct)$, where
$C=C(W,V)$ and $\mu(\tacka)$ satisfies for a.e. $t>0$
\begin{equation*}
|\mu'|^{2}(t)=\int_{\Omega} \left|P_{x}\left(-\nabla
 W*\mu(r)(x)-\nabla V(x)\right)\right|^{2} d\mu(t,x),
\end{equation*}
and for any $0\leq s\leq t<\infty$
\begin{equation*}
\mathcal{E}(\mu(s))=
\mathcal{E}(\mu(t))+\int_{s}^{t}\int_{\Omega}\left|P_{x}\left(-\nabla
 W*\mu(r)(x)-\nabla V(x)\right)\right|^{2}d\mu(r,x)dr.
\end{equation*}
Moreover if we have two such solutions $\mu^{i}(\tacka)$ with initial
data $\mu^{i}_{0}$ satisfying for $i=1,2$, $\supp(\mu^{i}_{0})$ are
compact and $\supp(\mu^{i}(t))\subset B\left(r(t)\right)$ for all
$t>0$, then for all $k\in\mathbb{N}$ such that $B(r(t))\subset
K_{k}$ we have
\begin{equation}\label{localcontra}
d_{W}\left(\mu^{1}(t),\mu^{2}(t)\right)\leq \exp\left(\left(
-\lambda_{W,k}^{-}-\lambda_{V,k} +\frac{\|\nabla
W\|_{L^{\infty}(\omok)}+\|\nabla
V\|_{L^{\infty}(\Omega_{k})}}{\eta}\right)t\right)d_{W}(\mu^{1}_{0},\mu^{2}_{0}).
\end{equation}
 where $\lambda_{W,k},\lambda_{V,k}$ are the geodesic convexity constants of $W$ and $V$ in
 $K_{k}$ and $\Omega_{k}=\Omega\cap K_{k}$.
\end{thm}

Let us point out that we are not able to get the system of Evolution Variational
Inequalities in its whole generality although they hold for compactly supported reference measures.

\begin{remk} \label{rem:instab}
Here we illustrate on an example that well-posedness of weak measure solutions cannot hold on domains which have an inside corner. 
Let $\Omega=\{(r\cos(\theta),r\sin(\theta))\in\mathbb{R}^{2}:0\leq r\leq 1, \frac{\pi}{4}\leq\theta\leq \frac{7\pi}{4}\}$ be as in Figure \ref{fig:instability}. 
Let $V(x)=-2 x_1$ be the external potential and $W$ be any $C^{2}$ convex  interaction potential with $\nabla W(0)=0$. Define $\gamma_1(s)=(1,-1) \, s$  and  $\gamma_2(s)=(1,1)\, s$ for $0\leq s\leq 1$. 
Then for initial datum $\mu_0 = \delta_0$ both $\mu_1(t) = \delta_{\gamma_1(t)}$ and 
 $\mu_2(t) = \delta_{\gamma_2(t)}$ are weak measure solutions. Thus uniqueness and hence stability of solutions cannot hold.

\begin{figure}[h]
\centering
\includegraphics[width=0.48\linewidth,angle=270]{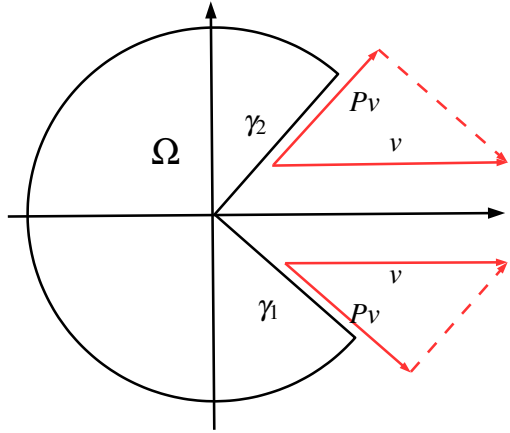}
\put(-170,-70){\Large $\Omega$}
\put(-80,-65){$v$}
\put(-95,-50){$Pv$}
\put(-80,-115){$v$}
\put(-95,-130){$Pv$}
\put(-130,-127){$\gamma_{1}$}
\put(-135,-57){$\gamma_{2}$}
\caption{The red arrows show the projected velocity field $Pv$ on $\gamma_{1}$ and $\gamma_{2}$, which are driving the particles apart from each other.}
\label{fig:instability}
\end{figure}
\end{remk}

\subsection{Strategy of the proof.}
The strategy to construct weak measure solutions to (\ref{conequ})
is to show the existence of gradient flow with respect to
$\mathcal{E}$. We approximate the initial data $\mu_0$ in Wasserstein metric by
$\mu^{n}_0=\sum_{i=1}^{k(n)} m^{n}_{i}\delta_{x^{n}_{i}}$ for
$x^{n}_{i}\in \Omega\bigcap B(n)$, and solve (\ref{conequ}) with
$\mu^{n}(0)=\mu^{n}_0$. Then (\ref{conequ}) becomes a discrete
projected system, for $1\leq i\leq k(n)$
\begin{equation}\label{disproj3}
  \left\{
   \begin{aligned}
   \dot{x}^{n}_{i}(t) &=P_{x^{n}_{i}(t)}\left(-\sum_{j}m^{n}_{j}\nabla
W\left(x^{n}_{i}-x^{n}_{j}\right)-\nabla V(x^{n}_{i})\right) \te{ a.e. } t\geq 0,  \\
   x^{n}_{i}(0) &= x^{n}_i\in \Omega , \\
   \end{aligned}
  \right.
  \end{equation}
which we show based on the well-posedness theory from non-convex sweeping process
differential inclusions with perturbations. For the general theory of sweeping processes we
refer to \cite{ET1,ET2,Ven} and references therein. To be precise, based on \cite{ET2} there exists a
locally absolutely continuous function $[0,\infty)\ni t\mapsto
x(t)=(x^{n}_{1}(t),\cdots,x^{n}_{k(n)}(t))\in \Omega^{n}=\Omega\times ...\times
\Omega$, such that for a.e. $t>0$,
\begin{equation}\label{disinclu}
-\dot{x}(t) \in N(\Omega^{n},x(t))-v(t,x(t)),
\end{equation}
where $v(t,x(t))=-\sum_{j}m^{n}_{j}\nabla
W\left(x^{n}_{i}-x^{n}_{j}\right)-\nabla V(x^{n}_{i})$ in our case. We then show that the solution to
(\ref{disinclu}) is actually a solution to (\ref{disproj3}).

Next we explore the properties of the sequence of solutions $\{\mu^{n}(\tacka)\}_{n}$. In particular,
\begin{itemize}
\item When $\Omega$ is bounded or $\Omega$ is unbounded but convex, we first prove the stability of
$\mu^{n}(t)$
\begin{equation}\label{sta3}
d_{W}(\mu^{n}(t),\mu^{m}(t))\leq \exp(C
t)d_{W}(\mu^{n}_{0},\mu^{m}_{0}),
\end{equation}
where $C=C(W,V)$ is a constant depending only on $W,V$. Thus
$\mu^{n}(t)$ converges to some $\mu(t)$ in $\mathcal{P}_{2}(\Omega)$
as $n\to\infty$. Since $\mu^{n}$ satisfies the energy dissipation
inequality,
\begin{equation*}
\mathcal{E}(\mu^{n}(s))\geq
\mathcal{E}(\mu^{n}(t))+\frac{1}{2}\int_{s}^{t}
|\left(\mu^{n}\right)'|^{2}(r)dr +\frac{1}{2}\int_{s}^{t}
\int_{\Omega} |P_{x}\left(-\nabla W*\mu^{n}(r)(x)-\nabla
V(x)\right)|^{2}d\mu^{n}(r,x)dr,
\end{equation*}
by the lower semicontinuity property, we are able to show that
$\mu(\tacka)$ also satisfies the desired energy dissipation inequality
\begin{equation}\label{EDI2}
\mathcal{E}(\mu(s))\geq \mathcal{E}(\mu(t))+\frac{1}{2}\int_{s}^{t}
|\mu'|^{2}(r)dr +\frac{1}{2}\int_{s}^{t} \int_{\Omega}
|P_{x}\left(-\nabla W*\mu(r)(x)-\nabla V(x)\right)|^{2}d\mu(r,x)dr.
\end{equation}
We then show the chain rule, for $\tilde v(t)$ is the tangent
velocity of $\mu(\tacka)$ at time $t$
\begin{equation}\label{chainrule3}
\frac{d}{dt}\mathcal{E}(\mu(t))=\int_{\Omega} \left\langle
-P_{x}\left(-\nabla W*\mu(t)(x)-\nabla V(x)\right), \tilde
v(t,x)\right\rangle d\mu(t,x),
\end{equation}
which together with the energy dissipation inequality yields that
$\mu(\tacka)$ is a gradient flow with respect to $\mathcal{E}$ and a weak
measure solution to (\ref{conequ}).
\item When $\Omega$ is unbounded and only $\eta$-prox-regular, we first show that the support of
the solutions $\mu^{n}(t)$ grows at most exponentially, i.e.
\begin{equation}\label{expgrowth3}
\supp\left(\mu^{n}(t)\right)\subset B(r(t)),
\end{equation}
for $r(t)=(r_0+1)\exp(Ct)$ given that $\supp(\mu_0)\subset B(r_0)$.
We then show that, given $\supp\left(\mu^{n}(t)\right)$ has the same
growth condition for all $n\in \mathbb{N}$, $\mu^{n}(\tacka)$ still
converges to a locally absolutely continuous curve $\mu(\tacka)$ satisfying
(\ref{EDI2}) and (\ref{chainrule3}). Thus $\mu(\tacka)$ is a weak measure
solution to (\ref{conequ}).
\end{itemize}

\subsection{Outline}
The paper is organized as follows. 

In Section \ref{sec:em}, we show the properties of the projection $P$ and then give the existence results for the discrete projected systems (\ref{disproj3}).

In Section \ref{sec:bddomain}, under the assumption that $\Omega$ is
bounded, we prove the stability of solutions to the discrete
projected systems $\mu^{n}(\tacka)$, i.e. (\ref{sta3}). Thus $\mu^{n}(\tacka)$
converge to an absolutely continuous curve $\mu(\tacka)$. We show that $\mu(\tacka)$ is
curve of maximum slope for the energy $\mathcal{E}$ and moreover a gradient flow solution of
(\ref{conequ}). We then show that $\mu(\tacka)$ is also a weak measure solution and that weak measure solutions satisfy the stability property (\ref{contra3}). At the end of the section, we show that solutions are characterized  by the system of Evolution Variational Inequalities (\ref{EVI3}).

Section \ref{sec:unbddglob} addresses the case of unbounded, convex
$\Omega$ and general initial data $\mu_{0}\in \mathcal{P}_{2}(\Omega)$, that is
Theorem \ref{globexis}.  The proof of Theorem \ref{globexis} is
similar to Theorem \ref{exgrad} and Theorem \ref{stagrad}, we
only concentrate on the key
differences.

Section \ref{sec:unbddomain} is devoted to the case  when $\Omega$ is
unbounded and only $\eta$-prox-regular with $\supp(\mu_{0})$ compact. We show that the support
of the solutions to the discrete projected systems (\ref{disproj3})
satisfy exponential growth condition (\ref{expgrowth3}). By similar
stability results as in Section \ref{sec:bddomain}, $\mu^{n}(\tacka)$ still
converges to a locally absolutely continuous curve $\mu(\tacka)$ and $\mu(\tacka)$
is a solution to (\ref{conequ}) with the desired energy dissipation
(\ref{EDI2}). We then give the proof of the stability result
(\ref{localcontra}) for
solutions with control on growth of supports. We end the section by making
a remark about well-posedness of (\ref{conequ}) with
time-dependent potentials $W, V$.

In the last Section \ref{sec:aggre}, we prove Proposition \ref{single_contr} and discuss the conditions on the shape of the 
domain $\Omega$ such that there exist interaction potentials $W$ for which 
solutions $\mu(\tacka)$ of (\ref{conequ}) aggregate to a singleton (a single delta mass).  

\section{Existence of solutions to discrete systems} \label{sec:em}
In this section, we first show properties of the projection $P$,
in particular the lower semicontinuity and convexity property of $P$.
Then we give the existence result of solutions to the discrete
projected systems (\ref{disproj3}).\\

Recall that the tangent and normal cones $T(\Omega,x)$ and  $N(\Omega,x)$
are closed convex cones by Definition \ref{def:Ccones}.
\begin{prop}\label{condecomp}
Suppose $\Omega$ satisfies (M1) and $x\in\partial\Omega$. Then for
any $v\in\mathbb{R}^{d}$, there exist a unique orthogonal decomposition
$(v_{T},v_{N})\in T(\Omega,x)\times N(\Omega,x)$ of $v$:
$$\langle v_{T},v_{N}\rangle =0 \; \te{ and } \; v=v_{T}+v_{N}.$$
Moreover, $v_{T}=\proj_{T(\Omega,x)}(v)=P_{x}(v),
v_{N}=\proj_{N(\Omega,x)}(v)$.
\end{prop}
Proposition \ref{condecomp} is a direct consequence of Moreau's
decomposition theorem, see \cite{M,R} for the proof.

\begin{prop}\label{lscproj}
Assume $\Omega$ satisfies (M1), then the map $\Omega\times
\mathbb{R}^{d}\ni(x,v)\mapsto |P_{x}(v)|^{2}$ is lower
semicontinuous and for any fixed $x\in \Omega$, $\mathbb{R}^{d}\ni
v\mapsto |P_{x}(v)|^{2}$ is convex.
\end{prop}
\begin{proof}
We first show the lower semicontinuity property. Let
$\{x_{n}\}_{n}\subset \Omega, \{v^{n}\}_{n}\subset \mathbb{R}^{d}$
be such that $\lim_{n\to\infty}x_{n}=x\in \Omega, \lim_{n\to\infty}
v^{n}=v$. If $x_{n}\in \mathring\Omega$ for all $n$ sufficiently
large, then $P_{x_{n}}\left(v^{n}\right)=v^{n}$ and we have
$|P_{x}(v)|^{2}\leq |v|^{2}=\lim_{n\to \infty} |v^{n}|^{2}$. And for
any $x\in \mathring\Omega$, we have $x_{n}\in\mathring\Omega$ for
$n$ sufficiently large, thus
$$\liminf_{n\to\infty}|P_{x_n}\left(v^{n}\right)|^{2}\geq
|P_{x}(v)|^{2}.$$ So we only need to check for $x\in \partial\Omega$
and $\{x_n\}_{n}\subset\partial\Omega$ such that
$\lim_{n\to\infty}x_n=x$. Denote the decomposition of $v^{n}$ as in
Proposition \ref{condecomp} by $$v^{n}=v^{n}_{T}+v^{n}_{N}$$ where
$v^{n}_{T}\in T(\Omega,x_n), v^{n}_{N}\in N(\Omega,x_n)$ and
$\langle v^{n}_{T}, v^{n}_{N}\rangle =0$. For any subsequence, which
we do not relabel, such that there exists $w_{N}\in\mathbb{R}^{d}$
and $\lim_{n\to\infty}v^{n}_{N}=w_{N}$, we claim that $w_{N}\in
N(\Omega,x)$ and $\langle v-w_{N},w_{N}\rangle=0$. Indeed, since
$\Omega$ is $\eta$-prox-regular,
$$B_{\eta} \left(x_n+\eta\frac{v^{n}_{N}}{|v^{n}_{N}|} \right) \cap \Omega=\emptyset.$$
Taking $n\to\infty$ implies
$$B_{\eta} \left(x+\eta\frac{w_{N}}{|w_{N}|} \right)\cap \Omega=\emptyset,$$ which then
implies $w_{N}\in N(\Omega,x)$. Also by taking $n\to \infty$ in
$\langle v^{n}-v^{n}_{N},v^{n}_{N}\rangle =0$ we get $\langle
v-w_{N},w_{N}\rangle=0$. We then know
\begin{align*}
|P_{x}(v)|^{2} & =|v_{T}|^{2}\\ & =|v-v_{N}|^{2} \\ &\leq |v-w_{N}|^{2}\\
&=\lim_{n\to\infty}|v^n-v^{n}_{N}|^{2}\\
& =\lim_{n\to\infty}|P_{x_n}\left(v^n\right)|^{2}
\end{align*}
So $$\liminf_{n\to\infty} |P_{x_n}\left(v^n\right)|^{2}\geq
|P_{x}(v)|^{2}.$$ We turn to the convexity property. For any fixed
$x\in\Omega$, if $x\in \mathring\Omega$ then $P_{x}(v)=v$ for all
$v\in \mathbb{R}^{d}$ and $v\mapsto |v|^{2}$ is convex. Now for
fixed $x\in \partial\Omega$, and any $v^1,v^2\in \mathbb{R}^{d},
0\leq \theta\leq 1$, denote the unique projection of $v^1, v^2$
defined in Proposition \ref{condecomp} by
$$v^{i}=v^{i}_{N}+v^{i}_{T}$$ for $i=1,2$. Then
$$(1-\theta)v^1+\theta v^2= \left((1-\theta)v^{1}_{T}+\theta
v^{2}_{T}\right)+\left((1-\theta) v^{1}_{N}+\theta
v^{2}_{N}\right).$$ Note that $(1-\theta)v^{1}_{T}+\theta
v^{2}_{T}\in T(\Omega,x)$ and $(1-\theta)v^{1}_{N}+\theta
v^{2}_{N}\in N(\Omega,x)$, by Proposition \ref{condecomp} we have
\begin{align*}
|P_{x}\left((1-\theta)v^1+\theta v^2\right)|^{2} & \leq
|(1-\theta)v^{1}_{T}+\theta v^{2}_{T}|^{2}\\ & \leq
(1-\theta)|v^{1}_{T}|^{2}+\theta |v^{2}_{T}|^{2} \\ & =
(1-\theta)|P_{x}\left(v^1\right)|^{2}+\theta
|P_{x}\left(v^2\right)|^{2}.
\end{align*}
Convexity is verified.
\end{proof}

We cite the following result from \cite{ET1,ET2} about the existence
of differential inclusions
\begin{thm}\label{exdiffinclu}
Assume that $S$ is $\eta$-prox-regular as defined in Definition \ref{def:proxset} and $F:\mathbb{R}^{d}\ni
x\mapsto F(x)\in\mathbb{R}^{d}$ is a continuous function with at
most linear growth, i.e., there exists some constant $C>0$ such that
$$|F(x)|\leq C(1+|x|).$$ Then the differential inclusion
\begin{equation}\label{diffinclu}
  \left\{
   \begin{aligned}
   -\dot{x}(t) &\in N(S,x(t))+F(x(t)) \te{ a.e. } t\geq 0,  \\
   x(0) &= x_0\in S.  \\
   \end{aligned}
  \right.
  \end{equation}
  has at least one locally absolutely continuous solution.
\end{thm}
Note that the theorems, for example Theorem 5.1 from \cite{ET2}, are
more general than Theorem \ref{exdiffinclu}. However, we only need
the simplified version for our purpose. We also notice that
(\ref{diffinclu}) implies that $x(t)\in S$ for all $t\geq 0$.
Indeed, since $N(S,x)=\emptyset$ for all $x\not\in S$ we know
$x(t)\in S$ for a.e. $t\geq 0$. Then the continuity of $x(t)$ and
the fact that $S$ is closed imply that $x(t)\in S$ for all $t\geq
0$. For completeness, we give a sketch of proof here.
\begin{proof}
For $T<\frac{1}{2C}$ where $C$ is constant in the growth condition
of $F$. For $n\in \mathbb{N}$, take the partition
$0=t^{n}_{0}<t^{n}_{1}<...<t^{n}_{n}= T$ and define
$\delta^{n}_{i}=t^{n}_{i+1}-t^{n}_{i}, x^{n}_{0}=x_0,
Z^{n}_{0}=F(x^{n}_0)$. Then define iteratively for $0\leq i\leq n-1$
\begin{equation*}
x^{n}_{i+1}=\proj_{S}\left(x^{n}_{i}-\delta^{n}_{i}Z^{n}_{i}\right)
\end{equation*}
and
\begin{equation*}
Z^{n}_{i+1}=F(x^{n}_{i+1}).
\end{equation*}
Note that we have then
\begin{equation*}
\|x^{n}_{i+1}\|\leq \|x^{n}_{i}\|+2\delta^{n}_{i}\|Z^{n}_{i}\|
\end{equation*}
and
\begin{equation*}
\|Z^{n}_{i}\|\leq C\left(1+\|x^{n}_{i}\|\right).
\end{equation*}
Thus
\begin{align*}
\|x^{n}_{i+1}\| & \leq \|x_{0}\|+\sum_{j=0}^{i} 2\delta^{n}_{j}C(1+\|x^{n}_{j}\|)\\
& \leq \|x_0\|+2CT(1+\max_{0\leq j\leq i}\|x^{n}_{j}\|),
\end{align*}
which implies
\begin{equation*}
\max_{0\leq i\leq n}\|x^{n}_{i}\|\leq \|x_0\|+ 2CT (1+\max_{0\leq
i\leq n}\|x^{n}_{i}\|).
\end{equation*}
Since $2CT<1$ we have uniformly in $n$
\begin{equation*}
\max_{0\leq i\leq n}\|x^{n}_{i}\|\leq
\frac{\|x_0\|+2CT}{1-2CT}<\infty,
\end{equation*}
and \begin{equation*}
 \max_{0\leq i\leq n}\|Z^{n}_{i}\|\leq
C(1+\max_{0\leq i\leq n}\|x^{n}_{i}\|)<\infty.
\end{equation*}
We now define the approximation solution by
\begin{equation*}
x_{n}(t)=u^{n}_{i}+\frac{x^{n}_{i+1}-x^{n}_{i}+\delta^{n}_{i}Z^{n}_{i}}{\delta^{n}_{i}}-(t-t^{n}_{i})Z^{n}_{i},
\end{equation*}
for $t^{n}_{i}\leq t<t^{n}_{i+1}$. Notice that $x_n$ can also be
written as
\begin{equation}\label{intform}
x_{n}(t)=x_0+\int_{0}^{t}[\Pi_{n}(s)-Z_{n}(s)]ds
\end{equation}
where
\begin{equation*}
\Pi_{n}(t)=\sum_{i=0}^{n}\frac{
x^{n}_{i+1}-x^{n}_{i}+\delta^{n}_{i}Z^{n}_{i}}{\delta^{n}_{i}}\chi_{(t^{n}_{i},t^{n}_{i+1}]}(t)
\end{equation*}
and $Z_{n}(t)=Z^{n}_{i}$ for $t^{n}_{i}\leq t<t^{n}_{i+1}$. We have
for a.e. $t\in [t^{n}_{i},t^{n}_{i+1})$
\begin{equation*}
\dot{x}_{n}(t)+Z_{n}(t)=\Pi_{n}(t)\in N\left(S,
x_{n}(t^{n}_{i})\right).
\end{equation*}
Since $\|\Pi_{n}(t)\|\leq \|Z^{n}_{i}\|$ for $t\in
(t^{n}_{i},t^{n}_{i+1}]$, we know there exists a subsequence of $n$,
which we do not relabel, such that
$$\Pi_{n}\rightharpoonup\Pi,\quad  Z_{n}\rightharpoonup Z\qquad \te{as } n \to \infty $$ 
weakly in $L^{2}[0,T]$. We then have by (\ref{intform}) that $x_{n}$ converges
locally uniformly to $x$ with
\begin{equation*}
x(t)=x_0+\int_{0}^{t}[\Pi(s)-Z(s)]ds.
\end{equation*}
We now claim that $x(t)$ is a solution to the differential inclusion
on $[0,T]$. First we check that $x(t)\in S$ for all $t\in [0,T]$.
Since
\begin{equation*}
\|x_{n}(t^{n}_i)-x(t)\|\leq \|x_{n}(t)-x(t)\|+c|t^{n}_{i}-t|,
\end{equation*}
$x(t)=\lim_{n\to\infty} x_{n}(t^{n}_{i})\in S$.
We then verify that $\dot{x}(t)+Z(t)\in - N(S,x(t))$ for a.e. $t\in
[0,T]$. Since $\dot{x}_{n}+Z_{n}=\Pi_{n}\rightharpoonup \Pi$ weakly
in $L^{2}\left([0,T]\right)$ and $\Pi_{n}(t)\in
N\left(S,x_{n}(t^{n}_{i})\right)$ for $t^{n}_{i}<t\leq t^{n}_{i+1}$,
by Mazur's lemma, for a.e. $t\in [0,T]$
\begin{equation*}
\dot{x}(t)+Z(t)\in \bigcap_{n}\{\dot{x}_{k}(t)+Z_{k}(t):k\geq n\}.
\end{equation*}
Then by Proposition 2.1 from \cite{ET2}, we know for a.e. $t\in
[0,T]$,
\begin{equation*}
\dot{x}(t)+Z(t)\in N(S,x(t)).
\end{equation*}
Now we only need to check that $Z(t)=F(x(t))$. We know that
$Z_{n}(t)=F(x^{n}(t^{n}_{i})$ for $t^{n}_{i}\leq t<t^{n}_{i+1}$.
Define $\tilde u_{n}$ by $\tilde x_{n}(t)=x^{n}(t^{n}_{i})$ for
$t^{n}_{i}\leq t<t^{n}_{i+1}$ and note
$Z_{n}(t)=F(x^{n}(t^{n}_{i})=F(\tilde x_{n}(t))$. Then $\tilde
x_{n}$ converges locally uniformly to $x$. Together with the fact
that $F$ is continuous, $F(\tilde x_{n})$ converges to $F(x)$ in
$L^{2}\left([0,T]\right)$. Since it is direct to check $Z_{n}$
converges weakly to $Z$ in $L^{2}\left([0,T]\right)$, we get
$Z(t)=F(x(t))$ for a.e. $t\in [0,T]$. The claim is proved.
\end{proof}
 We now show
that the solutions for the differential inclusions are actually
solutions for the projected systems.
\begin{lem}\label{inclu=proj}
Assume that $S$ is $\eta$-prox-regular by Definition \ref{def:proxset} and $x(t)$ is a locally
absolutely continuous solution to the differential inclusion
(\ref{diffinclu}). Then
\begin{equation}
\dot{x}(t)=P_{x(t)}\left(-F\left(x(t)\right)\right) \te{ a.e. }
t\geq 0.
\end{equation}
\end{lem}
\begin{proof}
Since $S$ is $\eta$-prox-regular, it is tangentially regular,
that is
$$T(S,x)=K(S,x)$$
where $T(S,x)$ is defined in Definition \ref{def:Ccones} and $K(S,x)$ is the contingent cone defined as
$$K(S,x)=\{v\in\mathbb{R}^{d}: \exists t_n\searrow 0 \;  \exists
v_{n}\to v \te{ s.t. }  (\forall n) \;  x+t_{n}v_{n}\in S\}.$$
We refer to \cite{Bou} for the details. Now note that for a.e. $t$
$$\dot{x}(t)=\lim_{h\to 0^{+}} \frac{x(t+h)-x(t)}{h}\in K(S,x(t))$$
and
$$\dot{x}(t)=\lim_{h\to 0^{-}}\frac{x(t+h)-x(t)}{h}\in -K(S,x(t)).$$
Thus  $\langle\dot{x}(t),n(x(t))\rangle =0$ for any
$n(x(t))\in N(S,x(t))$. From the differential inclusion
(\ref{diffinclu}),we know that $-F(x(t))=\dot{x}(t)+n(x(t))$ for
some $n(x(t))\in N(S,x(t))$. Together with fact that $\dot{x}(t)\in
T(S,x(t))$ and $\langle \dot{x}(t), n(x(t))\rangle =0$, by
Proposition \ref{condecomp}
\begin{equation*}
\dot{x}(t)=P_{x(t)}\left(-F\left(x(t)\right)\right),
\end{equation*}
as claimed.
\end{proof}
We turn to the existence of solutions to the discrete projected
system \eqref{disproj3}, which we write as
\begin{equation}
  \left\{
   \begin{aligned}
   \dot{x}_{i}(t) & = P_{x_{i}(t)}\left(v(x^{n}(t))\right),  \\
   x_{i}(0) &= x_{i}\in S.  \\
   \end{aligned}
  \right.
  \end{equation}
for $i=1,\cdots,n$. For that purpose we apply Theorem \ref{exdiffinclu}
and  Lemma \ref{inclu=proj}
 for $S=\Omega^{n}$ and
$x_0=(x_{1},\cdots,x_{n})$ with
$F(x)=\left(-v_{1}(x(t)),\cdots,-v_{n}(x(t))\right)$, where
$v_{i}(x(t))=-\nabla W*\mu(t)(x_{i}(t))-\nabla
V(x_{i}(t))=-\sum_{j=1}^{n} m_j \nabla
W\left(x_{i}(t)-x_{j}(t)\right)-\nabla V(x_{i}(t))$. To do that, we
first check that $\Omega^n$ is $\eta$-prox-regular.
\begin{prop}
If $\Omega\subset\mathbb{R}^{d}$ is $\eta$-prox-regular by Definition \ref{def:proxset}, then
$$\Omega^{n}=\{(x_1,\cdots,x_n)\, :\; x_{i}\in \Omega, \;\, i = 1, \dots,n \}$$ is
$\eta$-prox-regular; Also for any $x=(x_1,\cdots,x_n)\in \Omega^{n}$ we
have $$N(\Omega^{n},x)=N(\Omega,x_1)\times\cdots\times N(\Omega,x_n).$$
\end{prop}
\begin{proof}
To see $\Omega^{n}$ is also $\eta$-prox-regular, first it is direct
that $\Omega^{n}$ is a closed set. Now for any $x=(x_1,\cdots,x_n)\in
\partial \Omega^{n}$ and $v=(v^1,\cdots,v^n)\in N(\Omega^{n},x)$, by
Definition \ref{def:Ccones} there exists $\alpha>0$ such that
$$x\in P_{\Omega^{n}}\left(x+\alpha v\right),$$ which implies
$$x_i\in P_{\Omega}\left(x_i+\alpha v^i\right)$$ for $1\leq i\leq
n$. By the equivalent definition of $\eta$-prox-regularity of $\Omega$
(\ref{proxequiv}), we then have $$\langle
v^i, y_i-x_i\rangle \leq \frac{|v^{i}|}{2\eta}|y_i-x_i|^{2}$$ for
any $y_{i}\in \Omega$. Thus
\begin{align*}
\langle v,y-x\rangle & = \sum_{i=1}^{n}\langle v^i,y_i-x_i\rangle \\
& \leq \sum_{i=1}^{n} \frac{|v^{i}|}{2\eta}|y_i-x_i|^{2}\\ & \leq
\frac{|v|}{2\eta}|y-x|^{2},
\end{align*}
for any $y=(y_1,\cdots,y_n)\in \Omega^{n}$. Thus $\Omega^{n}$ is
$\eta$-prox-regular by (\ref{proxequiv}). We now turn to the relations between the normal
cones. For $x=(x_1,\cdots,x_n)\in\Omega^{n}$ and $v=(v^1,\cdots,v^n)$
\begin{align*}
v\in N(\Omega^{n},x) & \Leftrightarrow \exists \alpha>0\te{ s.t. }
x\in P_{\Omega^{n}}\left(x+\alpha v\right)\\ & \Leftrightarrow
x_{i}\in P_{\Omega}\left(x_i+\alpha v^i\right),\;\; i=1, \dots, n
\\ & \Leftrightarrow v_{i}\in N(\Omega,x_i), \;\; i=1, \dots, n .
\end{align*}
Thus $N(\Omega^n,x)=N(\Omega,x_1)\times\cdots\times N(\Omega,x_n)$.
\end{proof}
Now we give the main result regarding the existence of solutions to
projected discrete systems.
\begin{thm} \label{existode}
Assume that $\Omega$ is $\eta$-prox-regular by Definition \ref{def:proxset}. If either $\Omega$ is
bounded and $W,V$ satisfy (A1)-(A2) or $\Omega$ is unbounded and
$W,V$ satisfy (GA2) and (GA4) i.e. (LA2) and (LA4), then for any $n\in\mathbb{N}$ and any
$(x_1,\cdots,x_n)\in \Omega^{n},(m_1,\cdots,m_n)\in \mathbb{R}^{n}$ with
$m_i\geq 0,\sum_{i=1}^{n}m_i=1$, the projected discrete system
\begin{equation}
  \left\{
   \begin{aligned}
   \dot{x}_{i}(t) &=P_{x_{i}(t)}\left(v_{i}(x(t))\right),  \\
   x_{i}(0) &= x_{i}\in \Omega,  \\
   \end{aligned}
  \right.
  \end{equation}
for $i=1,\cdots,n$, where $v_{i}(x(t))=-\nabla W*\mu(t)(x_i(t))-\nabla
V(x_i(t))=-\sum_{j=1}^{n}m_{j}\nabla W(x_i(t)-x_j(t))-\nabla
V(x_i(t))$, has a locally absolutely continuous solution.
\end{thm}
\begin{proof}
We just need to check the conditions for Theorem \ref{exdiffinclu}
to apply. We already know that $\Omega^{n}$ is $\eta$-prox-regular.
If $\Omega$ is bounded and $W,V$ satisfy (A1)-(A2), then the
mapping $\Omega^{n}\ni y=(y_1,\cdots,y_n)\mapsto F(y)=\left(\nabla
W*\mu(y_1)+\nabla V(y_1),\cdots,\nabla W*\mu(y_n)+\nabla V(y_n)\right)$ where
$\mu=\sum_{i=1}^{n}m_i\delta_{y_i}$, is continuous and bounded.
Extend $F$ to $\mathbb{R}^{d n}$ so that $F$ is still continuous
and bounded. Then by Theorem \ref{exdiffinclu} there exists an
absolutely continuous solution to the differential inclusion
\begin{equation}\label{disproj}
  \left\{
   \begin{aligned}
   -\dot{x}(t) &\in N(\Omega^{n},x(t))+F(x(t)),  \\
   x(0) &= (x_1,\cdots,x_n)\in \Omega^{n}.  \\
   \end{aligned}
  \right.
  \end{equation}
Similarly, if $\Omega$ is unbounded and $\nabla W,\nabla V$ satisfy liner growth conditions (GA2) and (GA4),
then the mapping $\mathbb{R}^{d n}\ni y=(y_1,\cdots,y_n)\mapsto
F(y)=\left(\nabla W*\mu(y_1)+\nabla(y_1),\cdots,\nabla W*\mu(y_n)+\nabla
V(y_n)\right)$ where $\mu=\sum_{i=1}^{n}m_i\delta_{y_i}$, is continuous and
has linear growth on $\mathbb{R}^{d n}$. By Theorem
\ref{exdiffinclu}, we still have an absolutely continuous solution
to (\ref{disproj}).\\
Now consider (\ref{disproj}) in components yields for $1\leq i\leq
n$ and $v_{i}(x)=-\sum_{j=1}^{n}\nabla W(x_i-x_j) m_{j}-\nabla
V(x_i)$,
\begin{equation*}
  \left\{
   \begin{aligned}
   -\dot{x}_{i}(t) &\in N(\Omega,x_{i}(t))-v_{i}(x(t)),  \\
   x_{i}(0) &= x_{i}\in \Omega.  \\
   \end{aligned}
  \right.
  \end{equation*}
Then similar argument as in Lemma \ref{inclu=proj} gives
\begin{equation*}
  \left\{
   \begin{aligned}
   \dot{x}_{i}(t) &=P_{x_{i}(t)}\left(v_{i}(x(t))\right),  \\
   x_{i}(0) &= x_{i}\in \Omega.  \\
   \end{aligned}
  \right.
  \end{equation*}

\end{proof}

\section{Existence and stability of solutions with $\Omega$ bounded} \label{sec:bddomain}
In this section, we show the existence and stability of solutions to (\ref{conequ}) for the case
when $\Omega$ is bounded, prox-regular and $W,V$ satisfy (A1)-(A2).\\

We approximate $\mu_{0}\in\mathcal{P}_{2}(\Omega)$ by
$\mu^{n}_{0}=\sum_{i=1}^{k(n)}m^{n}_i\delta_{x^{n}_{i}}$ such that
$x^{n}_{i}\in\Omega$ and $\lim_{n\to\infty}
d_{W}\left(\mu_{0},\mu^{n}_{0}\right)=0$. By Theorem \ref{existode}, for each $n\in\mathbb{N}$
there exists a a locally absolutely continuous solution to
\begin{equation} \label{ode_sys3}
  \left\{
   \begin{aligned}
   \dot{x}^{n}_{i}(t) &=P_{x^{n}_{i}(t)}\left(v^{n}_{i}(x(t))\right), \quad 1\leq i\leq k(n) \\
   x^{n}_{i}(0) &= x^{n}_{i}\in \Omega, \\
   \end{aligned}
  \right.
  \end{equation}
for $t\geq 0$,
where
 $$v^{n}_{i}\left(x(t)\right)=-\nabla W*\mu^{n}(t)(x^{n}_{i}(t))-\nabla V(x^{n}_{i}(t))=-\sum_{j=1}^{k(n)}m^{n}_{j}\nabla W\left(x^{n}_{i}(t)-x^{n}_{j}(t)\right)-\nabla V\left(x^{n}_{i}(t)\right)$$ and
$\mu^{n}(t)=\sum_{j=1}^{k(n)}m^{n}_{j}\delta_{x^{n}_{j}(t)}$. It is a
straightforward calculation to see that for any $\phi\in
C_{c}^{\infty}(\mathbb{R}^{d})$
\begin{equation*}
\frac{d}{dt}\int_{\mathbb{R}^{d}}\phi(x)d\mu^{n}(t,x)=\int_{\mathbb{R}^{d}}
\left\langle \nabla\phi(x),P_{x}\left(v^{n}(t,x)\right)\right\rangle
d\mu^{n}(t,x).
\end{equation*}
Thus $\mu^{n}(t)$ satisfies
\begin{equation*}
\partial_{t}\mu^{n}(t,x)+\dive\left(\mu^{n}(t,x)P_{x}\left(v^{n}(t,x)\right)\right)=0,
\end{equation*}
in the sense of distributions for $v^{n}(t,x)=-\nabla
W*\mu^{n}(t)(x)-\nabla V(x)$.

The following proposition contains the key estimate on the stability of solutions in the discrete case.
In particular it shows how the stability in Wasserstein metric $d_{W}$ defined in (\ref{def:Wassmetric}) is affected by the lack of convexity of the domain.
\begin{prop}\label{contrprop}
Assume that $\Omega$ is bounded and satisfies (M1), $W, V$ satisfy
(A1) and (A2). Then for two solutions $\mu^{n}(\tacka)$ and $\mu^{m}(\tacka)$ to the discrete system with
different initial data $\mu^{n}_{0},\mu^{m}_{0}$, we have for all
$t\geq 0$
\begin{equation}\label{contra4}
d_{W}\left(\mu^{n}(t),\mu^{m}(t)\right)\leq
\exp\left(\left(-\lamw^{-}-\lamv +\frac{\|\nabla
W\|_{L^{\infty}(\omo)}+\|\nabla
V\|_{L^{\infty}(\Omega)}}{\eta}\right)t\right)
d_{W}\left(\mu^{n}_{0},\mu^{m}_{0}\right).
\end{equation}
\end{prop}
\begin{proof}
Note that $\mu^{n}(\tacka)$ is solution to the continuity equation
\begin{equation} \label{contn}
\partial_{t}\mu^{n}(t,x)+\dive\left(\mu^{n}(t,x)P_{x}\left(v^{n}(t,x)\right)\right)=0,
\end{equation}
for $v^{n}(t,x)=-\nabla W*\mu^{n}(t)(x)-\nabla V(x)$.
Since the discrete solutions may have different numbers of particles we use a transportation plan to relate them. Let $\gamma_{t}\in \Gamma_{o}\left(\mu^{n}(t),\mu^{m}(t)\right)$ be the optimal plan between $\mu^{m}$ and $\mu^{n}$ defined in (\ref{def:Optplan}).
By Theorem 8.4.7 and Lemma 4.3.4 from \cite{AGS}
\begin{equation}\label{gendiff1}
\frac{1}{2}\frac{d}{dt} d_{W}^{2}\left(\mu^{n}(t),\mu^{m}(t)\right)
= \int_{\Omega}\left\langle
P_{x}\left(v^{n}(t,x)\right)-P_{y}\left(v^{m}(t,y)\right),x-y\right\rangle
d\gamma_{t}(x,y).
\end{equation}
We first establish the contractivity the solutions would have if the boundary conditions were not present and then account for the change due to velocity projection at the boundary.
For $v^{n},v^{m}$, by (A1) and (A2), that is the convexity of $W$ and $V$,
\begin{flalign}
\nonumber &\int_{\Omega\times \Omega}\langle v^{n}(t,x)-v^{m}(t,y),x-y\rangle
d\gamma_{t}(x,y)\\ \nonumber & = \int_{\Omega\times\Omega}\langle -\nabla
W*\mu^{n}(t)(x)-\nabla V(x)+\nabla W*\mu^{m}(t)(y)-\nabla
V(y),x-y\rangle d\gamma_{t}(x,y)\\ \nonumber & =
\frac{1}{2}\int_{\Omega\times\Omega}\int_{\Omega\times\Omega}\langle
-\nabla W(x-z)+\nabla W(y-w),x-y-z+w\rangle
d\gamma_{t}(z,w)d\gamma_{t}(x,y)\\ \nonumber
& + \int_{\Omega\times\Omega}\langle -\nabla V(x)+\nabla
V(y),x-y\rangle d\gamma_{t}(x,y)\\ & \leq
-\frac{1}{2}\lamw\int_{\Omega\times\Omega}\int_{\Omega\times\Omega}|x-z-y+w|^{2}d\gamma_{t}(z,w)d\gamma_{t}(x,y)
-\lamv\int_{\Omega\times\Omega} |x-y|^{2}d\gamma_{t}(x,y) \label{WConv} \\ \nonumber &
\leq
\left(-\lamw^{-}-\lamv\right)\int_{\Omega\times\Omega}|x-y|^{2}d\gamma_{t}(x,y)\\ \nonumber
& =
\left(-\lamw^{-}-\lamv\right)d_{W}^{2}\left(\mu^{n}(t),\mu^{m}(t)\right).\nonumber
\end{flalign}

For the boundary effect, by the fact that $\Omega$ is $\eta$-prox-regular we have (\ref{proxequiv}), thus
\begin{equation}\label{errorest}
\begin{aligned}
& \int_{\Omega\times\Omega} \langle
P_{x}\left(v^{n}(t,x)\right)-v^{n}(t,x)-P_{y}\left(v^{m}(t,y)\right)+v^{m}(t,y),x-y\rangle
d\gamma_{t}(x,y)\\ & \leq
\int_{\Omega\times\Omega}\frac{\|v^{n}(t)\|_{L^{\infty}(\Omega)}+\|v^{m}(t)\|_{L^{\infty}(\Omega)}}{2\eta}|y-x|^{2}d\gamma_{t}(x,y)
\\ & =
\frac{\|v^{n}(t)\|_{L^{\infty}(\Omega)}+\|v^{m}(t)\|_{L^{\infty}(\Omega)}}{2\eta}d_{W}^{2}\left(\mu^{n}(t),\mu^{m}(t)\right).
\end{aligned}
\end{equation}
Notice that $v^{i}(x)=-\nabla W(x)*\mu^{i}(t)(x)-\nabla V(x)$
implies that for $i=n,m$ $$\|v^{i}\|_{L^{\infty}(\Omega)}\leq
\|\nabla W\|_{L^{\infty}(\omo)}+\|\nabla
V\|_{L^{\infty}(\Omega)}<\infty.$$ Plugging back into
(\ref{gendiff1}) we have
\begin{align*}
\frac{1}{2}\frac{d}{dt} d_{W}^{2}\left(\mu^{n}(t),\mu^{m}(t)\right)
& = \int_{\Omega}\left\langle
P_{x}\left(v^{n}(t,x)\right)-P_{y}\left(v^{m}(t,y)\right),x-y\right\rangle d\gamma_{t}(x,y)\\
& =\int_{\Omega\times\Omega}\left\langle
v^{n}(t,x)-v^{m}(t,y),x-y\right\rangle d\gamma_{t}(x,y)\\ & +
\int_{\Omega\times\Omega}\left\langle
P_{x}\left(v^{n}(t,x)\right)-v^{n}(t,x)-P_{y}\left(v^{m}(t,y)\right)+v^{m}(t,y),x-y\right\rangle
d\gamma_{t}(x,y)\\ & \leq\left( -\lamw^{-}-\lamv +\frac{\|\nabla
W\|_{L^{\infty}(\omo)}+\|\nabla
V\|_{L^{\infty}(\Omega)}}{\eta}\right)d_{W}^{2}\left(\mu^{n}(t),\mu^{m}(t)\right).
\end{align*}
By Gronwall's inequality, we know \eqref{contra4} for all $t\geq 0$.
\end{proof}

Since $n\to \infty, d_{W}(\mu^{n}_{0},\mu_{0})\to 0$, by Proposition \ref{contrprop} the solutions $\mu_n$ of \eqref{ode_sys3} form a Cauchy sequence in $n$, with respect to Wasserstein metric.
Thus
\begin{equation} \label{constmu}
\mu^{n}(t) \overset{d_W}{\longrightarrow} \mu(t) \;\;\te{ as } n \to \infty,
\end{equation}
for some $\mu(t) \in \mathcal P_2(\Omega)$.
\medskip

\begin{remk}
Our goal is to show that $\mu(\tacka)$ is a weak measure solution of \eqref{conequ}.
The most immediate idea would be to try to pass to limit directly in Definition \ref{def:wms}.
However note that since
 $P_x$ is not continuous in $x$ and thus the velocity field
 governing the dynamics is not continuous (at the boundary of $\Omega$).
 Given that $\mu_n$ converge to $\mu$ only in the weak topology of measures, the lack of continuity of
 velocities prevents us to directly pass to limit in the integral formulation given in  Definition \ref{def:wms}.
 To show that $\mu(\tacka)$ is a weak measure solution of \eqref{conequ} we use the theory of gradient flows in the spaces of probability measures
$\mathcal{P}_{2}(\Omega)$. Namely, we establish that $\mu(\tacka)$ satisfies the steepest descent property with respect to the total energy $\mathcal{E}$ defined in (\ref{def:totaene}) by showing $\mu^{n}(\tacka)$ satisfies such property and the property is stable under the weak topology of measures (convergence in the Wasserstein metric $d_{W}$).
\end{remk}

We turn to the introducing the elements of the theory of  gradient flows in the space of probability measures.
\begin{defi} \label{subdif}
Let $\mu\in\mathcal{P}_{2}(\Omega)$, a vector field $\xi$ on $\Omega$ is said to be in the
\emph{subdifferential} of $\mathcal{E}$ at $\mu$ if $\xi\in L^{2}(\mu)$,
i.e.
\begin{equation*}
\int_{\Omega} |\xi(x)|^{2}d\mu(x)<\infty,
\end{equation*}
and
\begin{equation*}
\mathcal{E}(\nu)-\mathcal{E}(\mu)\geq \inf_{\gamma\in
\Gamma_{o}(\mu,\nu)}\int_{\Omega\times\Omega} \langle
\xi(x),y-x\rangle
d\gamma(x,y)+o\left(d_{W}\left(\mu,\nu\right)\right)
\end{equation*}
for any $\nu\in \mathcal{P}_{2}(\Omega)$.
\end{defi}
Given a locally absolutely continuous curve $[0,\infty)\ni t\mapsto
\mu(t)\in \mathcal{P}_{2}(\Omega)$, for the metric derivative $|\mu'|$ defined in
(\ref{def:metrderi}) we have the following
\begin{thm}
There exists a unique Borel vector field $v(t)\in L^{2}(\mu(t))$
such that $\mu(\tacka)$ satisfies
\begin{equation*}
\partial_{t}\mu(t)+\dive\left(\mu(t)v(t)\right)=0
\end{equation*}
in the sense of distributions and $v(t)$ satisfies
\begin{equation*}
|\mu'|^{2}(t)=\int_{\Omega} |v(t,x)|^{2}d\mu(t,x),
\end{equation*}
for a.e. $t>0$.
\end{thm}
For the proof of the theorem, refer to Theorem 8.3.1 from
\cite{AGS}. We call the unique vector field $v(t)$ the tangent
velocity field and define
\begin{defi}
A locally absolutely continuous curve $[0,\infty)\ni t\mapsto
\mu(t)\in \mathcal{P}_{2}(\Omega)$ is a \emph{gradient flow} with
respect to the energy functional $\mathcal{E}$ if for a.e. $t>0$
\begin{equation*}
v(t)\in -\partial\mathcal{E}(\mu(t)),
\end{equation*}
where $v(t)$ is the tangent velocity field for $\mu(t)$.
\end{defi}

We now
show that $\mu(t)$ is a curve of maximal slope with respect to
$\mathcal{E}$.
\begin{thm}\label{maxslope}
$\mu(t)$ satisfies for any $0\leq s<t<\infty$
\begin{equation}\label{maxslope1}
\mathcal{E}\left(\mu(s)\right)\geq
\mathcal{E}\left(\mu(t)\right)+\frac{1}{2}\int_{s}^{t}|\mu'|^{2}(r)dr+\frac{1}{2}\int_{s}^{t}\int_{\Omega}|P_{x}\left(v(r,x)\right)|^{2}d\mu(r,x)dr,
\end{equation}
where $v(r,x)=-\int_{\Omega}\nabla W(x-y)d\mu(r,y)-\nabla V(x)$.
\end{thm}
Before proving the theorem, we need the following
\begin{lem}\label{lscslope}
Assume (M1) holds for $\Omega$ and $\nu^{n}\in
\mathcal{P}_{2}(\Omega)$ converges narrowly to $\nu\in
\mathcal{P}_{2}(\Omega)$ with
$\sup_{n}\int_{\Omega}|x|^{2}d\nu^{n}(x) <\infty$, then
\begin{equation}
\int_{\Omega}|P_{x}\left(v(x)\right)|^{2}d\nu(x)\leq
\liminf_{n\to\infty}\int_{\Omega}|P_{x}\left(v^{n}(x)\right)|^{2}d\nu^{n}(x),
\end{equation}
where $v^{n}(x)=-\int_{\Omega}\nabla W(x-y)d\nu^{n}(y)-\nabla V(x)$
and $v(x)=-\int_{\Omega}\nabla W(x-y)d\nu(y)-\nabla V(x)$.
\end{lem}
\begin{proof}
Similar argument as in Lemma 2.7 from \cite{CDFLS} yields that
$\nabla W*\nu^{n}$ converges weakly to $\nabla W*\nu$, i.e., for any
$\phi\in C^{0}_{b}\left(\mathbb{R}^{d}\right)$
\begin{equation*}
\lim_{n\to\infty}\int_{\Omega} \nabla W*\nu^{n}(x)\cdot
\phi(x)d\nu^{n}(x)=\int_{\Omega}\nabla W*\nu(x)\cdot \phi(x)d\nu(x).
\end{equation*}
Then by Proposition \ref{lscproj} we proved in Section \ref{sec:em} and Proposition 6.42 from
\cite{FL}, we know that there exist two sequences of bounded
continuous functions $a_{i},b_{i}$ such that for all
$x\in\Omega,v\in\mathbb{R}^{d}$
\begin{equation*}
|P_{x}\left(v\right)|^{2}=\sup_{i\in
\mathbb{N}}\left\{a_{i}(x)+b_{i}(x)\cdot v\right\}.
\end{equation*}
Thus
\begin{align*}
\liminf_{n\to \infty} \int_{\Omega}
|P_{x}\left(v^{n}(x)\right)|^{2}d\nu^{n}(x) & = \liminf_{n\to
\infty} \int_{\Omega} \sup_{i}\left\{a_{i}(x)+b_{i}(x)\cdot v^{n}(x)\right\}d\nu^{n}(x)\\
& \geq \liminf_{n\to\infty}
\int_{\Omega}\left(a_{i}(x)+b_{i}(x)\left(-\nabla W*\nu^{n}(x)-\nabla V(x)\right)\right)d\nu^{n}(x)\\
& = \int_{\Omega} \left(a_{i}(x)+b_{i}(x)\left(-\nabla
W*\nu(x)-\nabla V(x)\right)\right)d\nu(x).
\end{align*}
Taking supremum over $i\in\mathbb{N}$ and using Lebesgue's monotone
convergence theorem then gives
\begin{align*}
\liminf_{n\to \infty} \int_{\Omega}
|P_{x}\left(v^{n}(x)\right)|^{2}d\nu^{n}(x) & \geq
\sup_{i\in\mathbb{N}}\int_{\Omega}
\left(a_{i}(x)+b_{i}(x)\left(\nabla W*\nu(x)+\nabla V(x)\right)\right)d\nu(x)\\
& = \int_{\Omega} |P_{x}\left(v(x)\right)|^{2}d\nu(x).
\end{align*}

\end{proof}
We now start to prove the theorem
\begin{proof}[Proof of Theorem \ref{maxslope}]
We first show that the map $t\mapsto \mathcal{E}(\mu^{n}(t))$ is
locally absolutely continuous. Indeed, for $0\leq s<t<\infty$
\begin{flalign}\label{absoE}
& \left|\mathcal{E}(\mu(t))-\mathcal{E}(\mu(s))\right| \\ \nonumber & =
\left|\sum_{i=1}^{k(n)}m^{n}_{i}\left(V(x^{n}_{i}(t))-V(x^{n}_{i}(s))\right)+\frac{1}{2}\sum_{i,j=1}^{k(n)}m^{n}_{i}m^{n}_{j}
\left(W(x^{n}_{i}(t)-x^{n}_{j}(t))-W(x^{n}_{i}(s)-x^{n}_{j}(s))\right)\right|\\ \nonumber & \leq
\sum_{i=1}^{k(n)}m^{n}_{i}\left|V(x^{n}_{i})-V(x^{n}_{j})\right|+
\frac{1}{2}\sum_{i,j=1}^{k(n)}m^{n}_{i}m^{n}_{j}\left|W(x^{n}_{i}(t)-x^{n}_{j}(t))-W(x^{n}_{i}(s)-x^{n}_{j}(s))\right| \\ \nonumber & \leq
\sum_{i=1}^{k(n)}m^{n}_{i}\|\nabla V\|_{L^{\infty}(\Conv(\Omega))}|x^{n}_{i}(t)-x^{n}_{i}(s)|
+\sum_{i=1}^{k(n)}m^{n}_{i}\|\nabla W\|_{L^{\infty}(\Conv(\omo))}|x^{n}_{i}(t)-x^{n}_{i}(s)|\\ \nonumber
 &\leq
\left(\|\nabla V\|_{L^{\infty}(\Conv(\Omega))}+\|\nabla W\|_{L^{\infty}(\Conv(\omo))}\right)\sum_{i=1}^{k(n)}m^{n}_{i}|x^{n}_{i}(t)-x^{n}_{i}(s)|.\nonumber
\end{flalign}
Thus $t\mapsto\mathcal{E}(\mu(t))$ is locally absolutely continuous since $t\mapsto x^{n}_{i}(t)$ is locally absolutely continuous.

Since $\mu^{n}(\tacka)$ are solutions to the discrete systems, it is direct to calculate that
\begin{equation*}
\frac{d}{dt}\mathcal{E}(\mu^{n}(t))=-\int_{\Omega}|P_{x}(v^{n}(t,x))|^{2}d\mu^{n}(t,x),
\end{equation*}
and $|\left(\mu^{n}\right)'|^{2}(t)\leq \int_{\Omega}|P_{x}(v^{n}(t,x))|^{2}d\mu^{n}(t,x)$ for a.e. $t>0$.
Combining with the fact that $t\mapsto \mathcal{E}(\mu^{n}(t))$ is locally absolutely continuous then gives,
\begin{equation}\label{dismaxslope}
\mathcal{E}\left(\mu^{n}(s)\right)\geq
\mathcal{E}\left(\mu^{n}(t)\right)+\frac{1}{2}\int_{s}^{t}|\left(\mu^{n}\right)'|^{2}(r)dr+
\frac{1}{2}\int_{s}^{t}\int_{\Omega}|P_{x}\left(v^{n}(r,x)\right)|^{2}d\mu^{n}(r,x)dr.
\end{equation}
Note that $\Omega$ is bounded, $W, V\in C^{1}(\mathbb{R}^{d})$ and
$\lim_{n\to\infty}d_{W}\left(\mu^{n}(r),\mu(r)\right)=0$ for any
$0\leq r<\infty$, we get
\begin{equation*}
\lim_{n\to\infty}\mathcal{E}(\mu^{n}(r))=\mathcal{E}(\mu(r)).
\end{equation*}
Also by Lemma \ref{lscslope}, for any $0\leq r<\infty$
\begin{equation*}
\liminf_{n\to\infty}
\int_{\Omega}|P_{x}\left(v^{n}(r,x)\right)|^{2}d\mu^{n}(r,x)\geq
\int_{\Omega}|P_{x}\left(v(r,x)\right)|^{2}d\mu(r,x).
\end{equation*}
By Fatou's lemma, we then have
\begin{equation}
\liminf_{n\to\infty}
\int_{s}^{t}\int_{\Omega}|P_{x}\left(v^{n}(r,x)\right)|^{2}d\mu^{n}(r,x)dr\geq
\int_{s}^{t}\int_{\Omega}|P_{x}\left(v(r,x)\right)|^{2}d\mu(r,x)dr.
\end{equation}
We now claim that
\begin{equation}\label{lscmetricslope}
\liminf_{n\to\infty}\int_{s}^{t}|\left(\mu^{n}\right)'|^{2}(r)dr\geq
\int_{s}^{t}|\mu'|^{2}(r)dr.
\end{equation}
To see that, first notice that
$\sup_{n}\int_{s}^{t}|\left(\mu^{n}\right)'|^{2}(r)dr < \infty$, so
$|\left(\mu^{n}\right)'|\in L^{2}([s,t])$ and converges weakly in
$L^{2}([s,t])$ to some function $A$ as $n\to \infty$. We then have
for any $0\leq s\leq S\leq T\leq t<\infty$
\begin{align*}
d_{W}\left(\mu(S),\mu(T)\right) & = \lim_{n\to\infty}
d_{W}\left(\mu^{n}(S),\mu^{n}(T)\right)\\ & \leq
\liminf_{n\to\infty} \int_{S}^{T}|\left(\mu^{n}\right)'|(r)dr \\ & =
\int_{S}^{T}A(r)dr.
\end{align*}
Thus we have $$|\mu'|(r)\leq A(r)$$ for $s\leq r\leq t$, which then
implies
\begin{align*}
\int_{s}^{t}|\mu'|^{2}(r)dr & \leq \int_{s}^{t} A^{2}(r) dr \\
& \leq \liminf_{n\to\infty}\int_{s}^{t}
|\left(\mu^{n}\right)'|^{2}(r)dr.
\end{align*}
The claim is proved. Now take $n\to\infty$ in
(\ref{dismaxslope})gives
\begin{equation*}
\mathcal{E}\left(\mu(s)\right)\geq
\mathcal{E}\left(\mu(t)\right)+\frac{1}{2}\int_{s}^{t}|\mu'|^{2}(r)dr+\frac{1}{2}\int_{s}^{t}\int_{\Omega}|P_{x}\left(v(r,x)\right)|^{2}d\mu(r,x)dr,
\end{equation*}
as desired.
\end{proof}
Note that as a byproduct of the proof, we obtain that $\mu(\tacka)$ is a
locally absolutely continuous curve in $\mathcal{P}_{2}(\Omega)$. We
now show the proof of the main Theorem \ref{exgrad}
\begin{proof}[Proof of Theorem \ref{exgrad}]
Since $\mu(\tacka)\in \mathcal{P}_{2}(\Omega)$ is locally absolutely
continuous, by Theorem 8.31 from \cite{AGS}, there exists a unique
Borel vector field $\tilde v$ such that the continuity equation
\begin{equation}\label{conequ831}
\partial_{t}\mu(t)+\dive\left(\mu(t)\tilde v(t)\right)=0,
\end{equation}
in the sense of distributions, i.e., tested against all $\phi\in
C_{c}^{\infty}\left([0,\infty)\times\mathbb{R}^{d}\right)$, and
\begin{equation*}
\int_{\Omega}|\tilde v(t,x)|^{2}d\mu(t,x)=|\mu'|^{2}(t),
\end{equation*}
for a.e. $t\geq 0$. Then by Proposition 8.4.6 from \cite{AGS}, for
a.e. $t>0$
\begin{equation}\label{tanvel}
\lim_{h\to
0}\left(\pi^{1},\frac{1}{h}\left(\pi^2-\pi^1\right)\right)_{\sharp}\gamma_{t}^{h}=\left(\Id\times
\tilde v(t)\right)_{\sharp}\mu(t),
\end{equation}
in $\left(\mathcal{P}_{2}(\Omega),d_{W}\right)$ for any
$\gamma_{t}^{h}\in \Gamma_{o}(\mu(t),\mu(t+h))$. Here we also need the following
stronger convergence: Denote the disintegration of $\gamma_{t}^{h}$ with respect to $\mu(t)$ by $\nu^{h}_{x}$, then as $h\to 0$,
$\int_{\Omega} \frac{y-\cdot}{h}d\nu_{\cdot}^{h}(y)$ converges to the vector field $\tilde v(t,\cdot)$ weakly in $L^{2}(\mu(t))$.
The observation is that
\begin{align*}
\lim_{h\to 0} \left\|\int_{\Omega}\frac{y-\cdot}{h}d\nu^{h}_{\cdot}(y)\right\|_{L^{2}(\mu(t))}^{2} & =
\lim_{h\to 0} \int_{\Omega} \left|\int_{\Omega}\frac{y-x}{h}d\nu^{h}_{x}(y)\right|^{2} d\mu(t,x) \\ & \leq
\lim_{h\to 0} \int_{\Omega\times \Omega} \frac{|y-x|^{2}}{h^{2}}d\gamma^{h}_{t}(x,y)\\ & =
\lim_{h\to 0} \frac{d_{W}^{2}(\mu(t),\mu(t+h))}{h^{2}} \\ & <
\infty.
\end{align*}
Thus $\int_{\Omega}\frac{y-\cdot}{h}d\nu^{h}_{\cdot}(y)$ converges weakly in $L^{2}(\mu(t))$ to some vector field $\hat v(t,\cdot)$. This
together with (\ref{tanvel}) implies $\hat v=\tilde v$ and we have the weak $L^2(\mu(t))$ convergence of $\int_{\Omega}\frac{y-\cdot}{h}d\nu^{h}_{\cdot}(y)$ to $\tilde v(t)$ as stated.

We now claim the
following chain rule: for a.e. $t>0$
\begin{equation}
\frac{d}{dt}\mathcal{E}\left(\mu(t)\right)=\int_{\Omega}
\left\langle\nabla W*\mu(t)(x)+\nabla V(x),\tilde
v(t,x)\right\rangle d\mu(t,x).
\end{equation}
Indeed, we first notice that since $\mu(\tacka)$ is locally absolutely
continuous, $\mathcal{E}(\mu(\tacka))$ is also locally absolutely
continuous. To see that we have
\begin{align*}
|\mathcal{E}(\mu(t))-\mathcal{E}(\mu(s))|  & \leq \frac{1}{2}\left|
\int_{\Omega\times \Omega}W(x-y)d\mu(t,x)d\mu(t,y)-
\int_{\Omega\times \Omega}W(z-w)d\mu(s,z)d\mu(s,w)\right| \\ & +
\left|\int_{\Omega}V(x)d\mu(t,x)-\int_{\Omega}V(z)d\mu(s,z)\right| \\
& \leq \int_{\Omega\times\Omega}\left(\|\nabla
V\|_{L^{\infty}\left(\Conv\left(\Omega\right)\right)}+\|\nabla
W\|_{L^{\infty}\left(\Conv\left(\omo\right)\right)}\right)|x-z|d\gamma(x,z) \\
& \leq \left(\|\nabla
V\|_{L^{\infty}\left(\Conv\left(\Omega\right)\right)}+\|\nabla
W\|_{L^{\infty}\left(\Conv\left(\omo\right)\right)}\right)d_{W}\left(\mu(t),\mu(s)\right).
\end{align*}
Thus by the locally absolute continuity of $\mu(\tacka)$,
$\mathcal{E}(\mu(\tacka))$ is also locally absolutely continuous. Now for
any fixed $\mu,\nu\in \mathcal{P}_{2}(\Omega)$ and $\gamma\in
\Gamma_{o}(\mu,\nu)$, consider the function
\begin{align}\label{fnonincrease}
f(t)= &
\frac{W\left(t\left(x_1-x_2\right)-(1-t)\left(y_1-y_2\right)\right)-W(x_1-x_2)}{2t}\\
& + \frac{2V\left(tx_2+(1-t)y_2\right)-2V(x_2)}{2t}-
\frac{\lamv}{2}t |x_2-y_2|^{2}-\frac{\lamw}{2} t
\left(|x_1-y_1|^{2}+|x_2-y_2|^{2}\right).\nonumber
\end{align}
Due to (A1) and (A2), the $\lambda$-geodesic convexity of $W, V$, we
know $f$ is non-decreasing on $[0,1]$. So $f(1)\geq \liminf_{t\to
0^{+}} f(t)$. Integrating over $d\gamma(x_1,y_1)d\gamma(x_2,y_2)$
gives
\begin{align*}
\mathcal{E}(\nu) -\mathcal{E}(\mu) & = \int_{\Omega \times \Omega}
\int_{\Omega \times \Omega}
\frac{W(y_1-y_2)+2V(y_2)-W(x_1-x_2)-2V(x_2)}{2}
d\gamma(x_1,y_1)d\gamma(x_2,y_2)\\
 \geq &  \int_{\Omega \times \Omega}\int_{\Omega \times \Omega}\left \langle \nabla W(x_2-x_1)+\nabla
V(x_2),y_2-x_2\right\rangle d\gamma(x_1,y_1)d\gamma(x_2,y_2)+o\left(d_{W}\left(\mu,\nu\right)\right)\\
= &  \int_{\Omega \times \Omega} \left\langle \int_{\Omega}\nabla
W(x_2-x_1)d\mu(x_1)+\nabla
V(x_2),y_2-x_2\right\rangle d\gamma(x_2,y_2)+o\left(d_{W}\left(\mu,\nu\right)\right)\\
=&  \int_{\Omega \times \Omega} \left\langle\nabla W*\mu(x_2)+\nabla
V(x_2),y_2-x_2\right\rangle
d\gamma(x_2,y_2)+o\left(d_{W}\left(\mu,\nu\right)\right).
\end{align*}
Denote $v(t,x)=-\nabla W*\mu(t,x)-\nabla V(x)$, we notice that
\begin{align*}
\langle -v(t,x_2),y_2-x_2\rangle & =\left\langle
-P_{x_2}\left(v(t,x_2)\right),y_2-x_2\right\rangle+\left\langle
-v(t,x_2)+P_{x_2}\left(v(t,x_2)\right),y_2-x_2\right\rangle\\ & \geq
\left\langle -P_{x_2}\left(v(t,x_2)\right),y_2-x_2\right\rangle
-\frac{\|\nabla W\|_{L^{\infty}(\omo)}+\|\nabla
V\|_{L^{\infty}(\Omega)}}{2\eta}|y_2-x_2|^{2},
\end{align*}
and
\begin{equation*}
\int_{\Omega\times\Omega}|x_2-y_2|^{2}d\gamma(x_2,y_2)=d_{W}^{2}(\mu,\nu).
\end{equation*}
Thus
\begin{equation*}
\mathcal{E}(\nu) -\mathcal{E}(\mu) \geq \int_{\Omega \times \Omega}
\left\langle-P_{x_2}\left(-\nabla W*\mu(x_2)-\nabla
V(x_2)\right),y_2-x_2\right\rangle
d\gamma(x_2,y_2)+o\left(d_{W}\left(\mu,\nu\right)\right),
\end{equation*}
which implies
\begin{equation*}
-P\left(v(t)\right)=-P\left(-\nabla W*\mu(t)-\nabla V\right)\in
\partial\mathcal{E}(\mu(t)).
\end{equation*}
Take $\mu=\mu(t),\nu=\mu(t+h)$ and
$\gamma^{h}_{t}\in\Gamma_{o}(\mu(t),\mu(t+h))$ then gives
\begin{flalign*}
& \lim_{h\to
0^{+}}\frac{\mathcal{E}\left(\mu(t+h)\right)-\mathcal{E}\left(\mu(t)\right)}{h}
\\ & \geq \limsup_{h\to 0^{+}}\left(\int_{\Omega \times \Omega} \left\langle\nabla
W*\mu(t,x_2)+\nabla V(x_2),\frac{y_2-x_2}{h}\right\rangle
d\gamma^{h}_{t}(x_2,y_2)+\frac{1}{h}o\left(d_{W}\left(\mu(t),\mu(t+h)\right)\right)\right)\\
& = \int_{\Omega} \left\langle \nabla W*\mu(t)(x_2)+\nabla
V(x_2),\tilde v(t,x_2)\right\rangle d\mu(t,x_2),
\end{flalign*}
where the last equality comes from (\ref{tanvel}). Similarly, by
taking $\mu=\mu(t), \nu=\mu(t-h)$, we have
\begin{equation*}
\lim_{h\to
0^{+}}\frac{\mathcal{E}\left(\mu(t)\right)-\mathcal{E}\left(\mu(t-h)\right)}{h}\leq
\int_{\Omega} \left\langle \nabla W*\mu(t)(x_2)+\nabla V(x_2),\tilde
v(t,x_2)\right\rangle d\mu(t,x_2).
\end{equation*}
Together with the fact that $\mathcal{E}(\mu(\tacka)$ is locally
absolutely continuous, we have for a.e. $t>0$
\begin{equation*}
\frac{d}{dt}\mathcal{E}\left(\mu(t)\right)=\int_{\Omega}
\left\langle\nabla W*\mu(t)(x)+\nabla V(x),\tilde
v(t,x)\right\rangle d\mu(t,x).
\end{equation*}
The claim is proved. Now for
$v_{N}(t,x)=v(t,x)-P_{x}\left(v(t,x)\right)$, we have $v_{N}(t,x)\in
N(\Omega,x)$ and $\|v_{N}(t)\|_{L^{\infty}(\Omega)}\leq
\|v(t)\|_{L^{\infty}(\Omega)}<\infty$. Thus
\begin{align*}
\lim_{h\to 0^{+}}\int_{\Omega\times\Omega}\left\langle
v_{N}(t,x),\frac{y-x}{h}\right\rangle d\gamma_{t}^{h}(x,y) & \leq
\lim_{h\to 0^{+}}\int_{\Omega\times\Omega}
\frac{\|v_{N}(t)\|_{L^{\infty}(\Omega)}}{2\eta}\frac{1}{h}|x-y|^{2}d\gamma^{h}_{t}(x,y)\\
& \leq \lim_{h\to
0^{+}}\frac{\|v(t)\|_{L^{\infty}(\Omega)}}{2\eta}\frac{d_{W}^{2}\left(\mu(t),\mu(t+h)\right)}{h}\\
& = 0,
\end{align*}
which together with the weak $L^{2}\left(\mu(t)\right)$-convergence of $\int_{\Omega}\frac{y-\cdot}{h}d\nu_{\cdot}(y)$ implies
\begin{equation*}
\int_{\Omega}\left\langle v_{N}(t,x),\tilde v(t,x)\right\rangle
d\mu(t,x)\leq 0.
\end{equation*}
 We then know that
\begin{equation}
\frac{d}{dt}\mathcal{E}\left(\mu(t)\right) \geq -\int_{\Omega}
\left\langle P_{x}\left(-\nabla W*\mu(t)(x)-\nabla
V(x)\right),\tilde v(t,x)\right\rangle d\mu(t,x).
\end{equation}
Together with (\ref{maxslope1}), we get for a.e $t>0$
\begin{equation*}
\tilde v(t,x)= P_{x}\left(v(t,x)\right)= P_{x}\left(-\nabla
 W*\mu(t)(x)-\nabla V(x)\right)\in -\partial\mathcal{E}\left(\mu(t)\right),
\end{equation*}
\begin{equation*}
|\mu'|^{2}(t)=\int_{\Omega}\left|P_{x}\left(-\nabla
 W*\mu(t)(x)-\nabla V(x)\right)\right|^{2}d\mu(t,x)
\end{equation*}
and for any $0\leq s\leq t<\infty$
\begin{equation*}
\mathcal{E}(\mu(s))=
\mathcal{E}(\mu(t))+\int_{s}^{t}\int_{\Omega}\left|P_{x}\left(-\nabla
 W*\mu(r)(x)-\nabla V(x)\right)\right|^{2}d\mu(r,x)dr.
\end{equation*}
Thus $\mu(\tacka)$ is a gradient flow with respect to $\mathcal{E}$ and
by (\ref{conequ831}), a weak measure solution to (\ref{conequ}).
\end{proof}
\begin{remk}
In \cite{CarLisMai}, Carrillo, Lisini and Mainini showed
weak $L^{2}(\mu(t))$ convergence of $\int_{\Omega}\frac{y-\cdot}{h}d\nu^{h}_{\cdot}(y)$
to $\tilde v(t,\cdot)$ in a more general setting than ours.
\end{remk}

We turn to the proof of Theorem \ref{stagrad}
\begin{proof}[Proof of Theorem \ref{stagrad}]
We show (\ref{contra3}) first. Let $\mu^{1}(\tacka),\mu^{2}(\tacka)$ be two
solutions to (\ref{conequ}), by Theorem 8.4.7 and Lemma 4.3.4 from
\cite{AGS}, we have
\begin{equation}\label{gendiff}
\frac{d}{dt}d_{W}^{2}\left(\mu^{1}(t),\mu^{2}(t)\right)=2\int_{\Omega\times\Omega}\langle
P_{x}\left(v^{1}(t,x)\right)-P_{y}\left(v^{2}(t,y)\right),x-y\rangle
d\gamma_{t}(x,y),
\end{equation}
where $\gamma_{t}\in \Gamma_{o}\left(\mu^{1}(t),\mu^{2}(t)\right)$
and $v^{i}(t,x)=-\nabla W*\mu^{i}(t)(x)-\nabla V(x)$ for $i=1,2$.
For $v^{i}$, by (A1) and (A2) similar argument as in the proof of
Proposition \ref{contrprop} gives
\[  \int_{\Omega\times \Omega}\langle v^{1}(t,x)-v^{2}(t,y),x-y\rangle
d\gamma_{t}(x,y) \leq
\left(-\lamw^{-}-\lamv\right)d_{W}^{2}\left(\mu^{1}(t),\mu^{2}(t)\right).
\]
By the fact that $\Omega$ is $\eta$-prox-regular we have
\begin{flalign*}
& \int_{\Omega\times\Omega} \langle
P_{x}\left(v^{1}(t,x)\right)-v^{1}(t,x)-P_{y}\left(v^{2}(t,y)\right)+v^{2}(t,y),x-y\rangle
d\gamma_{t}(x,y)\\ & \leq
\frac{\|v^{1}(t)\|_{L^{\infty}(\Omega)}+\|v^{2}(t)\|_{L^{\infty}(\Omega)}}{2\eta}d_{W}^{2}\left(\mu^{1}(t),\mu^{2}(t)\right),
\end{flalign*}
where $v^{i}$ satisfies $$\|v^{i}\|_{L^{\infty}(\Omega)}\leq
\|\nabla W\|_{L^{\infty}(\omo)}+\|\nabla
V\|_{L^{\infty}(\Omega)}<\infty.$$ Plugging back into
(\ref{gendiff}) yields
\begin{align*}
\frac{1}{2}\frac{d}{dt}d_{W}^{2}\left(\mu^{1}(t),\mu^{2}(t)\right) &
=\int_{\Omega\times\Omega}\langle
P_{x}\left(v^{1}(t,x)\right)-P_{y}\left(v^{2}(t,y)\right),x-y\rangle
d\gamma_{t}(x,y)\\ & = \int_{\Omega\times\Omega} \langle
v^{1}(t,x)-v^{2}(t,y),x-y\rangle d\gamma_{t}(x,y)\\ &
+\int_{\Omega\times\Omega}\langle
P_{x}\left(v^{1}(t,x)\right)-v^{1}(t,x)-P_{y}\left(v^{2}(t,y)\right)+v^{2}(t,y),x-y\rangle
d\gamma_{t}(x,y)\\ & \leq\left( -\lamw^{-}-\lamv +\frac{\|\nabla
W\|_{L^{\infty}(\omo)}+\|\nabla
V\|_{L^{\infty}(\Omega)}}{\eta}\right)d_{W}^{2}\left(\mu^{1}(t),\mu^{2}(t)\right).
\end{align*}
Then by Gronwall's inequality we have for all $t\geq 0$
\begin{equation}
d_{W}\left(\mu^{1}(t),\mu^{2}(t)\right)\leq
\exp\left(\left(-\lamw^{-}-\lamv +\frac{\|\nabla
W\|_{L^{\infty}(\omo)}+\|\nabla
V\|_{L^{\infty}(\Omega)}}{\eta}\right)t\right)
d_{W}\left(\mu^{1}_{0},\mu^{2}_{0}\right).
\end{equation}
(\ref{contra3}) is proved. For (\ref{EVI3}), we have if $\mu(\tacka)$ is
a weak measure solution to (\ref{conequ}), then for any
$\nu\in\mathcal{P}_{2}(\Omega)$ and
$\gamma_{t}\in\Gamma_{o}(\mu(t),\nu)$
\begin{align*}
\frac{1}{2}\frac{d}{dt}d_{W}^{2}\left(\mu(t),\nu\right) & =
\int_{\Omega\times\Omega}\left\langle
P_{x}\left(v(t,x)\right),x-y\right\rangle d\gamma_{t}(x,y)\\ & =
\int_{\Omega\times\Omega}\left(\langle v(t,x),x-y\rangle+\langle
P_{x}\left(v(t,x)\right)-v(t,x),x-y\rangle\right)d\gamma_{t}(x,y)\\
& \leq
\mathcal{E}(\nu)-\mathcal{E}(\mu(t))+\int_{\Omega\times\Omega}\left(-\frac{\lamw^{-}}{2}-\frac{\lamv}{2}+\frac{\|v(t)\|_{L^{\infty}(\Omega)}}
{2\eta}\right)|x-y|^{2} d\gamma_{t}(x,y)\\
& \leq
\mathcal{E}(\nu)-\mathcal{E}(\mu(t))+\left(-\frac{\lamw^{-}}{2}-\frac{\lamv}{2}+\frac{\|\nabla
W\|_{L^{\infty}(\omo)}+\|\nabla
V\|_{L^{\infty}(\Omega)}}{2\eta}\right)d_{W}^{2}\left(\mu(t),\nu\right).
\end{align*}
On the other hand, if $\mu^{1}(\tacka)$ satisfies (\ref{EVI3}) and
$\mu^{2}(\tacka)$ is the solution to (\ref{conequ}) such that
$\mu^{1}_0=\mu^{2}_0$ ,by Lemma 4.3.4 from \cite{AGS} we get
\begin{equation*}
\frac{1}{2}\frac{d}{dt}d_{W}^{2}\left(\mu^{1}(t),\mu^{2}(t)\right)
\leq \left( -\lamw^{-}-\lamv +\frac{\|\nabla
W\|_{L^{\infty}(\omo)}+\|\nabla
V\|_{L^{\infty}(\Omega)}}{\eta}\right)d_{W}^{2}\left(\mu^{1}(t),\mu^{2}(t)\right).
\end{equation*}
Again by Gronwall's inequality we have $\mu^{1}(t)=\mu^{2}(t)$ for
all $t\geq 0$. Thus the weak measure solution is characterized by
the system of evolution variational inequalities (\ref{EVI3}).
\end{proof}

\section{Existence and stability of solutions with $\Omega$
unbounded: Global case} \label{sec:unbddglob} 

In this section we
prove the existence and stability of (\ref{conequ}) with $\Omega$
unbounded, convex and $W,V$ satisfying (GA1)-(GA4).\\

For any initial data $\mu_{0}\in \mathcal{P}_{2}(\Omega)$ and fixed $x_0\in\Omega$, denote
$B_{n}(x_0)=\{x\in\mathbb{R}^{d}:|x-x_0|<n\}$, we
can take $\mu^{n}_{0}=\sum_{i=1}^{k(n)}m^{n}_{j}\delta_{x^{n}_{i}}$ for
$x^{n}_{i}\in \Omega\cap \overline{B_{n}(x_0)}$ and $\lim_{n\to\infty} d_{W}\left(\mu^{n}_{0},\mu_0\right)=0$. To see that, note
$\int_{\Omega}|x-x_0|^{2}d\mu_{0}(x)<\infty$, thus $\lim_{n\to\infty}\int_{\Omega\setminus B_{n}(x_0)}|x-x_0|^{2}d\mu_{0}(x)=0$.
For $\mu_{0}\lfloor_{\Omega\cap\overline{B_{n}(x_0)}}$, we can find $\tilde \mu^{n}_{0}$ composed of delta measures
with the same total mass as $\mu_{0}\lfloor_{\Omega\cap\overline{B_{n}(x_0)}}$, such that
$\supp(\mu^{n}_{0})\subset \Omega\cap \overline{B_{n}(x_0)}$ and
$\lim_{n\to\infty}d_{W}\left(\mu_{0}\lfloor_{\Omega\cap\overline{B_{n}(x_0)}},\tilde \mu^{n}_{0}\right)=0$. Then
$\mu^{n}_{0}=\tilde \mu^{n}_{0}+\left(1-\mu_{0}\left(\Omega\cap \overline{B_{n}(x_0)}\right)\right)\delta_{x_0}$ satisfies
the required conditions. Without loss of generality, we assume that $x_0=0\in\Omega$ and denote $B(n)=B_{n}(0)$.

As in Section \ref{sec:bddomain}, we first show the convergence of  $\mu^{n}(\tacka)$ as $n\to\infty$.
\begin{prop}\label{contrprop2}
Assume that $\Omega$ is unbounded and convex, $W, V$ satisfy
(GA1)-(GA4). Then for two solutions $\mu^{m}(\tacka),\mu^{n}(\tacka)$ to the discrete system with
different initial data $\mu^{m}_{0},\mu^{n}_{0}$, we have for all
$t\geq 0$
\begin{equation}\label{contra6}
d_{W}\left(\mu^{n}(t),\mu^{m}(t)\right)\leq \exp\left(-\left(\lamw^{-}+\lamv\right)
t\right) d_{W}\left(\mu^{n}_{0},\mu^{m}_{0}\right).
\end{equation}
\end{prop}
\begin{proof}
The proof is similar to the proof of Proposition
\ref{contrprop} once we notice that since $\Omega$ is
$\infty$-prox-regular, by (\ref{proxequiv}) for any $x,y \in \Omega$
\begin{equation*}
\left\langle
P_{x}\left(v^{n}(t,x)\right)-v^{n}(t,x),x-y\right\rangle \leq 0.
\end{equation*}
\end{proof}
So as $n\to\infty$ we again know that $\mu^{n}(t)$ converges to some
$\mu(t)\in\left(\mathcal{P}_{2}\left(\Omega\right),d_{W}\right)$.
Before proving that $\mu(t)$ is a curve of maximal slope, we need
\begin{prop}\label{lscwv}
Let $\mu_{n},\mu\in \mathcal{P}_{2}(\Omega)$ be such that
$\lim_{n\to\infty} d_{W}\left(\mu_{n},\mu\right)=0$ then
\begin{equation*}
\lim_{n\to\infty} \mathcal{V}(\mu_{n})=\mathcal{V}(\mu),
\end{equation*}
and
\begin{equation*}
\lim_{n\to\infty} \mathcal{W}(\mu_{n})=\mathcal{W}(\mu).
\end{equation*}
\end{prop}
\begin{proof}
Since the arguments are similar, it is enough for us to show the
property for $\mathcal{V}$. By (GA4),
there exists a constant $C>0$ such that $|V(x)|\leq C(1+|x|^{2})$. By Lemma 5.1.7 from
\cite{AGS}, since $V(x)+C|x|^{2}$ is lower semicontinuous and bounded from below,
\begin{equation*}
\liminf_{n\to\infty}\int_{\Omega}
\left(V(x)+C|x|^{2}\right)d\mu_{n}(x)\geq \int_{\Omega}
\left(V(x)+C|x|^{2}\right)d\mu(x).
\end{equation*}
$\lim_{n\to\infty} d_{W}\left(\mu_{n},\mu\right)=0$, we know
\begin{equation*}
\lim_{n\to\infty} \int_{\Omega} |x|^{2}d\mu_{n}(x)=\int_{\Omega}
|x|^{2}d\mu(x).
\end{equation*}
Thus
\begin{equation*}
\liminf_{n\to\infty}\int_{\Omega} V(x)d\mu_{n}(x)\geq \int_{\Omega}
V(x)d\mu(x).
\end{equation*}
Similarly, the condition
$C|x|^{2}-V(x)$ is lower semicontinuous and bounded from below implies
\begin{equation*}
\limsup_{n\to\infty} \int_{\Omega} V(x)d\mu_{n}(x)\leq
\int_{\Omega}V(x)d\mu(x).
\end{equation*}
Thus \begin{equation*} \lim_{n\to\infty}
\mathcal{V}(\mu_{n})=\mathcal{V}(\mu),
\end{equation*} as claimed.
\end{proof}
We estimate the growth of support of the solutions
$\mu^{n}(\tacka)$ to (\ref{disproj3}).
\begin{lem}\label{lem:suppgrowth}
Let $\mu^{n}_{0}$ be the approximation of $\mu_0$ such that
$\supp\left(\mu^{n}_{0}\right)\subset \Omega\cap B(n)$. Then
$\supp\left(\mu^{n}(t)\right)\subset \Omega\cap B(r(t))$ for $r(t)\leq
(n+1)\exp(C t)$ for some $C=C(W,V)$ independent of $n$.
\end{lem}
\begin{proof}
Define $r(t)=\sup_{i} |x^{n}_{i}(t)|$. For fixed $t>0$, assume that
$x^{n}_{i}(t)$ realizes $R(t)$ i.e., $r(t)=|x^{n}_{i}(t)|$, then
\begin{align*}
\left|\frac{d}{dt} \left|x^{n}_{i}\right|^{2}\right| & = 2
\left|\left\langle x^{n}_{i}(t),
P_{x^{n}_{i}}\left(-\sum_{j=1}^{k(n)}m_j\nabla
W(x^{n}_{i}-x^{n}_{j})-\nabla V(x^{n}_{i})\right)\right\rangle\right|\\
& \leq 2|x^{n}_{i}(t)|\left(\sum_{j=1}^{k(n)}m_j \left|\nabla
W(x^{n}_{i}(t)-x^{n}_{j}(t))\right|+\left|\nabla V(x^{n}_{i}(t))\right|\right)\\
& \leq 2|x^{n}_{i}(t)|\left(\sum_{j=1}^{k(n)}m_j
C\left(1+|x^{n}_{i}(t)+|x^{n}_{j}(t)|\right)+C\left(1+|x^{n}_{i}(t)|\right)\right)\\
& \leq C\left(1+\left|x^{n}_{i}(t)\right|^{2}\right).
\end{align*}
Thus $$r(t)\leq r(0)\exp(C t)+\exp(C t)-1$$ for $r(0)\leq n$ and $C$ depending only on
$W,V$, in particular independent of the number of particles $k(n)$.
\end{proof}

We can now show
\begin{thm}\label{thm:maxslope2}
Assume $\Omega$ is unbounded and convex, $W, V$ satisfy (GA1)-(GA4),
then $\mu(\tacka)$ satisfies for any $0\leq s<t<\infty$
\begin{equation}\label{maxslope4}
\mathcal{E}\left(\mu(s)\right)\geq
\mathcal{E}\left(\mu(t)\right)+\frac{1}{2}\int_{s}^{t}|\mu'|^{2}(r)dr+\frac{1}{2}\int_{s}^{t}\int_{\Omega}|P_{x}\left(v(r,x)\right)|^{2}d\mu(r,x)dr,
\end{equation}
where $v(r,x)=-\int_{\Omega}\nabla W(x-y)d\mu(r,y)-\nabla V(x)$.
\end{thm}
\begin{proof}
We first check that for fixed $n\in \mathbb{N}$, the function $t\mapsto \mathcal{E}(\mu^{n}(t))$ is locally absolutely
continuous. For fixed $0\leq s<t<\infty$, by Lemma \ref{lem:suppgrowth}, $\|\nabla V(x)\|_{L^{\infty}(\Omega\cap B(r(t)))}<\infty$ and
$\|\nabla W\|_{L^{\infty}(\Omega\cap B(r(t))-\Omega\cap B(r(t))}<\infty$. Then by the same argument as in (\ref{absoE}), $t\mapsto\mathcal{E}(\mu(t))$ is
locally absolutely continuous. Together with Proposition \ref{lscwv}, the proof is now identical to the
proof of Theorem \ref{dismaxslope}. We omit it here.
\end{proof}
We proceed to the proof of Theorem \ref{globexis}
\begin{proof}[Proof of Theorem \ref{globexis}] Let $\tilde v$ be the
tangential velocity field for $\mu(\tacka)$, i.e. $\mu(\tacka)$ satisfies
(\ref{conequ831}) and $\|\tilde v(t)\|_{L^{2}(\mu(t))}=|\mu'|(t)$.
Similar arguments as in the proof of Theorem \ref{exgrad} still
gives that for any $\mu,\nu\in \mathcal{P}_{2}(\Omega)$
\begin{equation*}
\mathcal{E}(\nu)-\mathcal{E}(\mu)\geq
\int_{\Omega\times\Omega}\left\langle \nabla W*\mu(x_2)+\nabla
V(x_2),y_2-x_2\right\rangle
d\gamma(x_2,y_2)+o\left(d_{W}(\mu,\nu\right),
\end{equation*}
and for a.e. $t>0$
\begin{equation*}
\frac{d}{dt}
\mathcal{E}\left(\mu(t)\right)=\int_{\Omega}\left\langle \nabla
W*\mu(t)(x)+\nabla V(x),\tilde v(t,x)\right\rangle d\mu(t,x).
\end{equation*}
Now since $\Omega$ is convex, we have $\langle v_{N}(t,x),y-x\rangle
\leq 0$, thus
\begin{equation*}
\mathcal{E}(\nu)-\mathcal{E}(\mu)\geq \int_{\Omega\times \Omega}
\left\langle -P_{x}\left(-\nabla W*\mu(t)(x)-\nabla V(x)\right),
y-x\right\rangle d\gamma(x,y),
\end{equation*}
and
\begin{equation*}
\lim_{h\to 0^{+}}\int_{\Omega\times\Omega}\left\langle
v_{N}(t,x),\frac{y-x}{h}\right\rangle d\gamma_{t}^{h}(x,y)\leq 0.
\end{equation*}
So we have $-P\left(v(t)\right)=-P\left(-\nabla W*\mu(t)-\nabla
V\right)\in \partial\mathcal{E}\left(\mu(t)\right)$ and
\begin{equation*}
\int_{\Omega}\langle v_{N}(t,x),\tilde v(t,x)\rangle d\mu(t,x)\leq
0.
\end{equation*}
Thus
\begin{equation*}
\frac{d}{dt} \mathcal{E}\left(\mu(t)\right)\geq
-\int_{\Omega}\left\langle P_{x}\left(-\nabla W*\mu(t)(x)-\nabla
V(x)\right),\tilde v(t,x)\right\rangle d\mu(t,x),
\end{equation*}
which together with Theorem \ref{maxslope2} implies for a.e. $t>0$
\begin{equation}
\tilde v(t,x)= P_{x}\left(v(t,x)\right)= P_{x}\left(-\nabla
 W*\mu(t)(x)-\nabla V(x)\right)\in -\partial\mathcal{E}\left(\mu(t)\right),
\end{equation}
\begin{equation}
|\mu'|^{2}(t)=\int_{\Omega}\left|P_{x}\left(-\nabla
 W*\mu(t)(x)-\nabla V(x)\right)\right|^{2}d\mu(t,x)
\end{equation}
and for any $0\leq s\leq t<\infty$
\begin{equation}
\mathcal{E}(\mu(s))=
\mathcal{E}(\mu(t))+\int_{s}^{t}\int_{\Omega}\left|P_{x}\left(-\nabla
 W*\mu(r)(x)-\nabla V(x)\right)\right|^{2}d\mu(r,x)dr.
\end{equation}
Thus $\mu(\tacka)$ is a gradient flow with respect to $\mathcal{E}$ and
by (\ref{conequ831}), a weak measure solution to (\ref{conequ}).\\
For the stability result (\ref{globcontra}), we only need to notice
that for any two solutions $\mu^{1}(\tacka),\mu^{2}(\tacka)$ to
(\ref{conequ}), since $\Omega$ is convex, $\left\langle
v^{i}(t,x)-P_{x}\left(v^{i}(t,x)\right),y-x\right\rangle \leq 0$ for
$i=1,2$. Thus
\begin{align*}
\frac{1}{2}\frac{d}{dt}d_{W}^{2}\left(\mu^{1}(t),\mu^{2}(t)\right) &
=\int_{\Omega\times\Omega}\langle
P_{x}\left(v^{1}(t,x)\right)-P_{y}\left(v^{2}(t,y)\right),x-y\rangle
d\gamma_{t}(x,y)\\ & = \int_{\Omega\times\Omega} \langle
v^{1}(t,x)-v^{2}(t,y),x-y\rangle d\gamma_{t}(x,y)\\ &
+\int_{\Omega\times\Omega}\langle
P_{x}\left(v^{1}(t,x)\right)-v^{1}(t,x)-P_{y}\left(v^{2}(t,y)\right)+v^{2}(t,y),x-y\rangle
d\gamma_{t}(x,y)\\ & \leq -\left(\lamw^{-}+\lamv\right)
d_{W}^{2}\left(\mu^{1}(t),\mu^{2}(t)\right).
\end{align*}
Then by Gronwall's inequality we get (\ref{globcontra}).\\ For
evolution variational inequalities (\ref{EVI4}), if $\mu(\tacka)$ is a
solution to (\ref{conequ}) then for any
$\nu\in\mathcal{P}_{2}(\Omega)$ and $\gamma\in
\Gamma_{o}\left(\mu(t),\nu\right)$ an optimal plan
\begin{align*}
\frac{1}{2}\frac{d}{dt}d_{W}^{2}\left(\mu(t),\nu\right) & =
\int_{\Omega\times\Omega}\left\langle
P_{x}\left(v(t,x)\right),x-y\right\rangle d\gamma_{t}(x,y)\\ & =
\int_{\Omega\times\Omega}\left(\langle v(t,x),x-y\rangle+\langle
P_{x}\left(v(t,x)\right)-v(t,x),x-y\rangle\right)d\gamma_{t}(x,y)\\
& \leq
\mathcal{E}(\nu)-\mathcal{E}(\mu(t))-\int_{\Omega\times\Omega}\left(\frac{\lamw^{-}}{2}+\frac{\lamv}{2}\right) |x-y|^{2}d\gamma_{t}(x,y)\\
& \leq \mathcal{E}(\nu)-\mathcal{E}(\mu(t))-\left(\frac{\lamw^{-}}{2}+\frac{\lamv}{2}\right)
d_{W}^{2}\left(\mu(t),\nu\right).
\end{align*}
On the other hand, if $\mu^{1}(\tacka)$ satisfies (\ref{EVI4}) and
$\mu^{2}(\tacka)$ is the solution to (\ref{conequ}) such that
$\mu^{1}_0=\mu^{2}_0$ we know that for any $\nu\in
\mathcal{P}_{2}(\Omega)$ and $i=1,2$
\begin{equation*}
\frac{1}{2}\frac{d}{dt}d_{W}^{2}\left(\mu^{i}(t),\nu\right)+\left(\frac{\lamw^{-}}{2}+\frac{\lamv}{2}\right)
d_{W}^{2}\left(\mu^{i}(t),\nu\right)\leq
\mathcal{E}(\nu)-\mathcal{E}\left(\mu^{i}(t)\right).
\end{equation*}
By Lemma 4.3.4 from \cite{AGS} we then have
\begin{equation*}
\frac{1}{2}\frac{d}{dt}d_{W}^{2}\left(\mu^{1}(t),\mu^{2}(t)\right)
\leq  -\left(\lamw^{-}+\lamv\right) d_{W}^{2}\left(\mu^{1}(t),\mu^{2}(t)\right).
\end{equation*}
So by Gronwall's inequality we have $\mu^{1}(t)=\mu^{2}(t)$ for all
$t\geq 0$. Thus the weak measure solution to (\ref{conequ}) is
characterized by the system of evolution variational inequalities
(\ref{EVI4}).
\end{proof}

\section{Existence and stability of solutions with $\Omega$ unbounded: Compactly supported initial data case} \label{sec:unbddomain}
In this section, we show the existence and stability results in the
case when $\Omega$ is unbounded and $W,V$ satisfy (LA1)-(LA4).
The novelty is that $\lambda$-geodesic convexity of energy is only assumed locally (which is automatically satisfied if $V$ and $W$ are $C^2$ functions).

We start by giving the control the support of the solutions
$\mu^{n}(t)$ to (\ref{disproj3}). Notice that when approximating
$\mu_0$ by $\mu^{n}_{0}=\sum_{i=1}^{k(n)}
m^{n}_{i}\delta_{x^{n}_{i}}$, since $\supp(\mu_0)\subset \Omega\cap B(r_0)$, we
can take $x^{n}_{i}\in \Omega\cap B(r_0+1)$ for all $n\in\mathbb{N}$ and $1\leq
i\leq k(n)$ such that we have
$$\lim_{n\to\infty} d_{W}\left(\mu^{n}_{0},\mu_{0}\right)=0.$$ So
without loss of generality, we assume
$\supp\left(\mu^{n}_{0}\right)\subset B(r_0)$ for all $n\in\mathbb{N}$. Then by
Lemma \ref{lem:suppgrowth}, $\supp\left(\mu^{n}(t)\right)\subset \Omega\cap B(r(t))$ for $r(t)\leq
(r_0+1)\exp(C t)$ for some $C=C(W,V)$ independent of $n$.

\begin{prop}
There exists a locally absolutely continuous curve $\mu(\tacka)$
 in $\mathcal{P}_{2}(\Omega)$ such that $\mu^{n}(t)$ converges
 to $\mu(t)$ in $\mathcal{P}_{2}(\Omega)$ for any $0\leq t<\infty$.
\end{prop}
\begin{proof}
For any fixed $0<T<\infty$ and any $0\leq t\leq T$, we know that
$\supp(\mu^{n}(t))\subset B(r(T))$ for all $0\leq t\leq T$ uniformly
in $n$. Let $K_{k}$ and $\lambda_{W,k}, \lambda_{V,k}$ be the sequences of compact
convex sets and constants such that $W, V$ are
$\lambda_{W,k}$ and $\lambda_{V,k}$-geodesically convex on $K_{k}$. Take $k_0$ be such that
$B(2r(T))\subset K_{k}$ for all $k\geq k_{0}$. Still denote
$\gamma_{t}\in\Gamma_{o}\left(\mu^{n}(t),\mu^{m}(t)\right)$ an optimal plan. Now
notice that $\supp(\mu^{n}(t)),\supp(\mu^{m}(t))\subset
B(r(t))\cap\Omega\subset K_{k}\cap\Omega=\Omega_{k}$, thus
\begin{equation*}
\int_{\Omega\times\Omega}\left\langle
v^{n}(t,x)-v^{m}(t,y),x-y\right\rangle d\gamma_{t}(x,y)\leq
\int_{\Omega\times\Omega} \left(\lambda_{W,k}^{-}+\lambda_{V,k}\right)|x-y|^{2},
\end{equation*}
and
\begin{flalign*}
& \int_{\Omega\times\Omega}\left\langle
P_{x}\left(v^{n}(t,x)\right)-v^{n}(t,x)-P_{y}\left(v^{m}(t,y)\right)+v^{m}(t,y),x-y\right\rangle
d\gamma_{t}(x,y)\\ & \leq
\int_{\Omega\times\Omega}\left(\frac{\|v^{n}(t)\|_{L^{\infty}(\Omega_{t})}+\|v^{m}(t)\|_{L^{\infty}(\Omega_{t})}}{2\eta}\right)
|x-y|^{2}d\gamma_{t}(x,y),
\end{flalign*}
where $\Omega_{t}=\Omega\cap B(r(t))$. Since
$v^{n}(t,x)=-\int_{\Omega}\nabla W(x-y)d\mu^{n}(t,y)-\nabla V(x)$ we
know
$$\|v^{n}(t)\|_{L^{\infty}(\Omega_{t})}\leq \|\nabla
W\|_{L^{\infty}(\Omega_{T}\!-\!\Omega_{T})}+\|\nabla
V\|_{L^{\infty}(\Omega_{T})}<\infty.$$ Thus as in the proof of
Proposition \ref{contrprop}, we have for $0\leq t\leq T$
\begin{align*}
\frac{1}{2}\frac{d}{dt} d_{W}^{2}\left(\mu^{n}(t),\mu^{m}(t)\right)
& = \int_{\Omega}\left\langle
P_{x}\left(v^{n}(t,x)\right)-P_{y}\left(v^{m}(t,y)\right),x-y\right\rangle d\gamma_{t}(x,y)\\
& =\int_{\Omega\times\Omega}\left\langle
v^{n}(t,x)-v^{m}(t,y),x-y\right\rangle d\gamma_{t}(x,y)\\ & +
\int_{\Omega\times\Omega}\left\langle
P_{x}\left(v^{n}(t,x)\right)-v^{n}(t,x)-P_{y}\left(v^{m}(t,y)\right)+v^{m}(t,y),x-y\right\rangle
d\gamma_{t}(x,y)\\ & \leq\left( -\lambda_{W,k}^{-}-\lambda_{V,k} +\frac{\|\nabla
W\|_{L^{\infty}(\Omega_{T}\!-\!\Omega_{T})}+\|\nabla
V\|_{L^{\infty}(\Omega_{T})}}{\eta}\right)d_{W}^{2}\left(\mu^{n}(t),\mu^{m}(t)\right).
\end{align*}
By Gronwall's inequality, we have for all $0\leq t\leq T$
\begin{equation*}
d_{W}\left(\mu^{n}(t),\mu^{m}(t)\right)\leq \exp\left(\left(
-\lambda_{W,k}^{-}-\lambda_{V,k} +\frac{\|\nabla
W\|_{L^{\infty}(\Omega_{T}\!-\!\Omega_{T})}+\|\nabla
V\|_{L^{\infty}(\Omega_{T})}}{\eta}\right)t\right)d_{W}(\mu^{n}_{0},\mu^{m}_{0}).
\end{equation*}
Thus as $n\to \infty$, $\mu^{n}(t)$ converges in
$\mathcal{P}_{2}(\Omega)$ to some $\mu(t)$.
\end{proof}
\begin{thm}
$\mu(\tacka)$ is a curve of maximal slope, for any $0\leq s<t<\infty$
\begin{equation}\label{maxslope2}
\mathcal{E}\left(\mu(s)\right)\geq
\mathcal{E}\left(\mu(t)\right)+\frac{1}{2}\int_{s}^{t}|\mu'|^{2}(r)dr+\frac{1}{2}\int_{s}^{t}\int_{\Omega}|P_{x}\left(v(r,x)\right)|^{2}d\mu(r,x)dr,
\end{equation}
where $v(r,x)=-\int_{\Omega}\nabla W(x-y)d\mu(r,y)-\nabla V(x)$.
\end{thm}
\begin{proof}
We use similar argument as in Theorem \ref{maxslope} and Theorem \ref{thm:maxslope2}. For any fixed $n\in\mathbb{N}$,
since $\supp(\mu^{n}(t))\subset \Omega\cap B(r(t))$, we can still control the $L^{\infty}$-norm of $\nabla V$ and $\nabla W$.
 Then the same argument as in the proof of
Theorem \ref{thm:maxslope2} shows that $t\mapsto \mathcal{E}\left(\mu(t)\right)$ is locally absolutely continuous.
Thus the fact that $\mu^{n}$ are solutions to the discrete
systems implies,
\begin{equation}\label{dismaxslope1}
\mathcal{E}\left(\mu^{n}(s)\right)\geq
\mathcal{E}\left(\mu^{n}(t)\right)+\frac{1}{2}\int_{s}^{t}|\left(\mu^{n}\right)'|^{2}(r)dr+
\frac{1}{2}\int_{s}^{t}\int_{\Omega}|P_{x}\left(v^{n}(r,x)\right)|^{2}d\mu^{n}(r,x)dr.
\end{equation}
Note that $W, V\in C^{1}(\mathbb{R}^{d})$ and
$\lim_{n\to\infty}d_{W}\left(\mu^{n}(r),\mu(r)\right)=0$ with
$\supp(\mu^{n}(r))\subset \Omega\cap B(r(T)$ for any $0\leq r<t\leq
T$, we get
\begin{equation*}
\lim_{n\to\infty}\mathcal{E}(\mu^{n}(r))=\mathcal{E}(\mu(r)).
\end{equation*}
By Lemma \ref{lscslope} and notice that $\nabla W*\mu^{n}(r)+\nabla
V$ still converges weakly to $\nabla W*\mu(r)+\nabla V$ for any
$0\leq r\leq T$, then
\begin{equation*}
\liminf_{n\to\infty}
\int_{\Omega}|P_{x}\left(v^{n}(r,x)\right)|^{2}d\mu^{n}(r,x)\geq
\int_{\Omega}|P_{x}\left(v(r,x)\right)|^{2}d\mu(r,x).
\end{equation*}
By Fatou's lemma,
\begin{equation*}
\liminf_{n\to\infty}
\int_{s}^{t}\int_{\Omega}|P_{x}\left(v^{n}(r,x)\right)|^{2}d\mu^{n}(r,x)dr\geq
\int_{s}^{t}\int_{\Omega}|P_{x}\left(v(r,x)\right)|^{2}d\mu(r,x)dr.
\end{equation*}
Now by the same argument as in the proof of (\ref{lscmetricslope}),
we again obtain
\begin{equation*}
\liminf_{n\to\infty}\int_{s}^{t}|\left(\mu^{n}\right)'|^{2}(r)dr\geq
\int_{s}^{t}|\mu'|^{2}(r)dr.
\end{equation*}
Take $n\to\infty$ in (\ref{dismaxslope1}) gives
\begin{equation*}
\mathcal{E}\left(\mu(s)\right)\geq
\mathcal{E}\left(\mu(t)\right)+\frac{1}{2}\int_{s}^{t}|\mu'|^{2}(r)dr+\frac{1}{2}\int_{s}^{t}\int_{\Omega}|P_{x}\left(v(r,x)\right)|^{2}d\mu(r,x)dr.
\end{equation*}
\end{proof}

We now start to prove Theorem \ref{locexistence}
\begin{proof}[Proof of Theorem \ref{locexistence}]
Since $\mu(\tacka)$ is locally absolutely continuous, we know that there
exists a unique Borel vector field $\tilde v$ such that
\begin{equation*}
\partial_{t}\mu(t)+\dive\left(\mu(t)\tilde v(t)\right)=0
\end{equation*}
holds in the sense of distributions. For a fixed $T>0$ and any
$\mu,\nu\in\mathcal{P}_{2}(\Omega)$ with
$\supp(\mu),\supp(\nu)\subset B(r(T))$, let
$\gamma\in\Gamma_{o}(\mu,\nu)$. Since $W, V$ are
$\lambda_{W,k}$ and $\lambda_{V,k}$-geodesically convex on $K_{k}\supset B(r(t))\cap\Omega$,
we have that the function $f$ we defined in (\ref{fnonincrease}) by
taking $\lambda=\lambda_{k}$ is non-decreasing in $t$ for any
$(x_1,y_1),(x_2,y_2)\in \supp\gamma$. Thus we still have
\begin{equation*}
\mathcal{E}(\nu)-\mathcal{E}(\mu)\geq
\int_{\Omega\times\Omega}\langle \nabla W*\mu(x_2)+\nabla
V(x_2),y_2-x_2\rangle d\gamma(x_2,y_2).
\end{equation*}
For any $0<t<T$, and $h>0$ such that $t-h\geq 0,t+h\leq T$, we take
$\mu=\mu(t),\nu=\mu(t+h)$ to get
\begin{equation*}
\lim_{h\to 0^{+}}\frac{\mathcal{E}(\mu(t+h))-\mathcal{E}(\mu(t))}{h}
\geq \int_{\Omega}\langle \nabla W*\mu(t)(x)+\nabla V(x),\tilde
v(t,x)\rangle d\mu(t,x)
\end{equation*}
Again take $\mu=\mu(t),\nu=\mu(t-h)$ gives
\begin{equation*}
\lim_{h\to
0^{+}}\frac{\mathcal{E}(\mu(t))-\mathcal{E}(\mu(t-h))}{h}\leq
\int_{\Omega} \langle \nabla W*\mu(t)(x)+\nabla V(x),\tilde
v(t,x)d\mu(t,x).
\end{equation*}
Also $\mathcal{E}(\mu(t))$ is locally absolutely continuous, so for
a.e. $t>0$
\begin{equation*}
\frac{d}{dt}\mathcal{E}(\mu(t))=\int_{\Omega}\langle \nabla
W*\mu(t)(x)+\nabla V(x),\tilde v(t,x)\rangle d\mu(t,x),
\end{equation*}
which again implies
\begin{equation*}
\frac{d}{dt}\mathcal{E}(\mu(t))\geq -\int_{\Omega}\langle
P_{x}\left(-\nabla W*\mu(t)(x)-\nabla V(x)\right),\tilde
v(t,x)\rangle d\mu(t,x).
\end{equation*}
Combine with (\ref{maxslope2}) yields
\begin{equation*}
\tilde v(t,x)= P_{x}\left(-\nabla W*\mu(t)(x)-\nabla V(x)\right),
\end{equation*}
and for any $0\leq s\leq t<\infty$
\begin{equation*}
\mathcal{E}(\mu(s))=
\mathcal{E}(\mu(t))+\int_{s}^{t}\int_{\Omega}\left|P_{x}\left(-\nabla
 W*\mu(r)(x)-\nabla V(x)\right)\right|^{2}d\mu(r,x)dr.
\end{equation*}
For the contraction property (\ref{localcontra}), we notice that for
any $0\leq t\leq T<\infty$ and $k\in\mathbb{N}$ such that
$B(r(T))\subset K_{k}$
\begin{equation*}
\frac{1}{2}\frac{d}{dt}d_{W}^{2}\left(\mu^{1}(t),\mu^{2}(t)\right)
\leq\left( -\lambda_{W,k}^{-}-\lambda_{V,k} +\frac{\|\nabla
W\|_{L^{\infty}(\omok)}+\|\nabla
V\|_{L^{\infty}(\Omega_{k})}}{\eta}\right)d_{W}^{2}\left(\mu^{1}(t),\mu^{2}(t)\right).
\end{equation*}
where $\Omega_{k}=\Omega\cap K_{k}$. Thus by Gronwall's inequality,
we have for all $0\leq t\leq T$
\begin{equation*}
d_{W}\left(\mu^{1}(t),\mu^{2}(t)\right)\leq \exp\left(\left(
-\lambda_{W,k}^{-}-\lambda_{V,k} +\frac{\|\nabla
W\|_{L^{\infty}(\omok)}+\|\nabla
V\|_{L^{\infty}(\Omega_{k})}}{\eta}\right)t\right)d_{W}(\mu^{1}_{0},\mu^{2}_{0}).
\end{equation*}
\end{proof}

\begin{remk}\label{timedep}
When the external and interaction potentials are time-dependent $V=V(t,x), W=W(t,x)$, then
with some modifications of the arguments we have before, we can still show the existence and stability results
of the solutions to (\ref{conequ}) in all the three different cases as in the time-independent settings before.
For example, we assume that there are constants $\lambda\in \mathbb{R}, \eta >0$ and
a positive function $\beta\in L^{1}([0,\infty))$ such that
\begin{itemize}
\item [\textbf{(M1)}] $\Omega$ is bounded and $\eta$-prox-regular.
\item [\textbf{(TA1)}] $W\in C^{1}([0,\infty)\times \mathbb{R}^{d})$ is $\lambda$-geodesically convex on $\Conv\left(\omo\right)$
uniformly in $t$.
\item [\textbf{(TA2)}] $V\in C^{1}([0,\infty)\times \mathbb{R}^{d})$ is $\lambda$-geodesically convex on $\Conv\left(\Omega\right)$
uniformly in $t$.
\item [\textbf{(TA3)}] $|\nabla V(t,x)|\leq \beta(t)(1+|x|)$ and $|\nabla W(t,x)| \leq \beta(t)(1+|x|)$ for all $x\in \mathbb{R}^{d}$.
\item [\textbf{(TA4)}] $|\frac{\partial V}{\partial t}(t,x)|\leq \beta(t)(1+|x|^{2})$ and $|\frac{\partial W}{\partial t}(t,x)|\leq \beta(t)(1+|x|^{2})$ for all $x\in \mathbb{R}^{d}$.
\end{itemize}
Then we can show the existence of a weak
measure solution $\mu(\tacka)$ to (\ref{conequ}) satisfying (\ref{metslo}), (\ref{enedis}) and stability estimate
\begin{equation}\label{dissta}
d_{W}(\mu^{1}(t),\mu^{2}(t))\leq \exp\left(-2\lambda t+\frac{C(\Omega)}{\eta}\int_{0}^{t}\beta(s) ds\right) d_{W}\left(\mu^{1}_{0},\mu^{2}_{0}\right),
\end{equation}
where $C(\Omega)=\sup_{x\in\Omega}\dist(x,0)$.
We sketch the proof and concentrate on the differences.
Approximate the initial data $\mu_0$ by a sequence of particle measures $\mu^{n}_0$ as before.
Note that we can still show the existence of solutions to the projected ODE system by citing Theorem 5.1 from \cite{ET2}.
Thus for total energy defined as $\mathcal{E}(t,\mu)=\frac{1}{2}\int_{\Omega\times\Omega} W(t,x-y)d\mu(x)d\mu(y)+\int_{\Omega} V(t,x)d\mu(x)$, we have the following energy dissipation along the solutions $\mu^{n}(\tacka)$,
\begin{equation}
\begin{split}
\mathcal{E}(s,\mu^{n}(s)) & \geq \mathcal{E}(t,\mu^{n}(t))-
\frac{1}{2}\int_{s}^{t}\int_{\Omega\times\Omega} \frac{\partial W}{\partial r}(r,x-y)d\mu(r,x)d\mu(r,y)d r\\
&  -\int_{s}^{t}\int_{\Omega} \frac{\partial V}{\partial r} (r,x)d\mu(r,x)dr + \frac{1}{2}\left|\left(\mu^{n}\right)'\right|^{2}+\frac{1}{2}\int_{\Omega} \left|P_{x}(v^{n}(t,x))\right|^{2}d\mu^{n}(t,x).
\end{split}
\label{enedis4}
\end{equation}

Similar stability argument as before shows that the sequence $\{\mu^{n}(\tacka)\}_{n}$ satisfies the stability estimate (\ref{dissta}).
Thus we know $\mu^{n}(\tacka)$ converges in $d_{W}$ to a locally absolutely curve $\mu(\tacka)$ and $\mu(\tacka)$ satisfies
the same energy dissipation (\ref{enedis4}) by similar lower semicontinuity arguments.
By the $\lambda$-geodesic convexity and $C^1$ regularity of $W$ and $V$, we can then show the following chain rule for $\mu(\tacka)$,
\begin{equation}
\begin{split}
\frac{d}{d t}\mathcal{E}(t,\mu(t)) & \geq \frac{1}{2}\int_{\Omega\times\Omega} \frac{\partial W}{\partial t}(t,x-y)d\mu(t,x)d\mu(t,y)
+ \int_{\Omega} \frac{\partial V}{\partial t}(t,x) d\mu(t,x)\\
& - \int_{\Omega} \left\langle P_{x}\left(v(t,x)\right),\tilde v(t,x)\right\rangle d\mu(t,x).
\end{split}
\label{chainrule5}
\end{equation}
Combining (\ref{enedis4}) with (\ref{chainrule5}), we show that $\mu(\tacka)$ is a weak measure solution to (\ref{conequ}) satisfying (\ref{metslo}) and
(\ref{enedis}). Then (\ref{dissta}) comes from the stability argument of the time-independent setting.
\end{remk}

\section{Aggregation on nonconvex domains}\label{sec:aggre}
In this section, we consider the following question: what are the conditions on $\Omega$ to ensure the existence of an interaction potential $W$ such that the 
solution $\mu(\tacka)$ to (\ref{conequ}) aggregates to a singleton (delta mass) as time goes to infinity? 

Let $\Omega$ be bounded and $\eta$-prox-regular, $V\equiv 0$ and $W$ satisfy (A1) for some $\lambda_{W}>0$, such that Theorem \ref{exgrad} holds and we have a weak measure solution $\mu(\tacka)$ to (\ref{conequ}).  We recall $\Xi=\{\delta_{x}: x\in\mathbb{R}^{d}\}$ the set of singletons, and start to estimate the evolution of $d_{W}(\mu(\tacka),\Xi)$, the distance of $\mu(\tacka)$ to $\Xi$. That is we prove Proposition \ref{single_contr}.

\begin{proof}
It suffices to show that for all $t \geq 0$
\[ \frac{1}{2}\frac{d^{+}}{dt}d_{W}^{2}\left(\mu(t),\Xi\right)\leq \left(-\lambda_{W}+\frac{\|\nabla
W\|_{L^{\infty}(\omo)}}{2\eta}\right)d_{W}^{2}\left(\mu(t),\Xi) \right. \]
since then by Gronwall's inequality the result follows.

By shifting time we can assume that $t=0$.
 Denote the center of mass for $\mu_0$ by $\bar x$, that is $\bar x=\int_{\Omega} x d\mu(0,x)$. It is direct computation to show that $d_{W}(\mu(0),\Xi)=d_{W}\left(\mu(0),\delta_{\bar x}\right),$ and for any $t>0$, $d_{W}(\mu(t),\Xi)\leq d_{W}\left(\mu(t),\delta_{\bar x}\right)$. Thus
\begin{align*}
\frac{1}{2}\frac{d^{+}}{dt}\bigg|_{t=0}d_{W}^{2}\left(\mu(t),\Xi\right) & \leq \frac{1}{2}\frac{d^{+}}{dt} \bigg|_{0} d_{W}^{2}\left(\mu(t),\delta_{\bar x}\right)\\ & =
\int_{\Omega}\left\langle P_{x}\left(v(0,x)\right),x-\bar x\right\rangle d\mu(0,x)\\ & = 
\int_{\Omega}\left( \left\langle v(0,x),x-\bar x\right\rangle+\left\langle P_{x}\left(v(0,x)\right)-v(0,x),x-\bar x\right\rangle \right)d\mu(0,x) . 
\end{align*}
Now we follow similar argument as in the proof of Proposition \ref{contrprop}. To be precise, by (\ref{WConv}) with $\mu^n(t) = \mu(t), \mu^m(t)\equiv \delta_{\bar x}$, we have 
\begin{align*}
\int_{\Omega}\left\langle v(0,x),x-\bar x\right\rangle d\mu(0,x) & \leq 
-\frac{\lambda_{W}}{2}\int_{\Omega\times \Omega} |x-y|^{2}d\mu(0,x)d\mu(0,y)\\   
& = -\lambda_{W}\int_{\Omega}|x-\bar x|^{2}d\mu(0,x) \\
& = -\lambda_{W} d_{W}^{2}\left(\mu(0),\delta_{\bar x}\right),
\end{align*}
where we used the fact that $\int_{\Omega} (x-\bar x)d\mu(0,x)=0$ for the definition of center of mass.

Also by (\ref{errorest}) with $\mu^n(t)=\mu(t), \mu^m(t)\equiv \delta_{\bar x}$,
\begin{equation*}
 \int_{\Omega\times\Omega} \langle
P_{x}\left(v(0,x)\right)-v(0,x),x-\bar x\rangle
d\mu(0,x)\leq \frac{\|\nabla
W\|_{L^{\infty}(\omo)}}{2\eta}d_{W}^{2}\left(\mu(0),\delta_{\bar x}\right).
\end{equation*}
Combine the estimates yields 
\begin{align*}
\frac{1}{2}\frac{d^{+}}{dt}\bigg|_{t=0}d_{W}^{2}\left(\mu(t),\Xi\right) & \leq \left(-\lambda_{W}+\frac{\|\nabla
W\|_{L^{\infty}(\omo)}}{2\eta}\right)d_{W}^{2}\left(\mu(0),\delta_{\bar x}\right)\\
& = \left(-\lambda_{W}+\frac{\|\nabla
W\|_{L^{\infty}(\omo)}}{2\eta}\right)d_{W}^{2}\left(\mu(0),\Xi \right).
\end{align*}
\end{proof}

We now prove Theorem \ref{aggrecri}.
\begin{proof}
It turns out that the quadratic interaction potential leads to the sharpest bound for general domains. Furthermore, since multiplying a potential by a positive constant only leads to a constant rescaling in time of the dynamics, we consider $W(x)=\frac{1}{2}|x|^{2}$. To verify the inequality (\ref{aggrecon}) note that $\nabla W(x)=x, \Hess W(x)=\Id$ and $\lambda_{W}=1$. Thus
$\sup_{\omo}|\nabla W|\leq \sup_{x,y\in\Omega}|x-y|=\Diam(\Omega)$ and 
$$-\lambda_{W}+\frac{\|\nabla W\|_{L^{\infty}\left(\omo\right)}}{2\eta}\leq -1+\frac{1}{2\eta}\Diam(\Omega) =: C(\Omega) <0$$
which via inequality (\ref{aggrecon}) implies the desired result.
 \end{proof}

\begin{remk}\label{sharpness}
We notice that (\ref{aggrecon}) implies that $\lim_{t\to\infty} d_{W}\left(\mu(t),\delta_{\bar x(t)}\right)=0$ where $\bar x(t)=\int_{\Omega} x d\mu(t,x)$ is the center of mass for $\mu(t)$. Hence as $t\to \infty$, $\mu(t)$ converges in $d_{W}$ to a singleton, i.e., all mass aggregates to one point to form a delta mass of size 1. Thus Theorem \ref{aggrecri} gives a sufficient condition on the shape of the domain $\Omega$ on which there exists a radially symmetric interaction potential $W$ so that solutions aggregate to a point. We note that the simple condition given in the theorem is also sharp in the following sense:  for any $\veps>0$ there exists $\Omega$ bounded and $\eta$-prox-regular with $0<\eta\leq (\frac{1}{2}-\epsilon)\Diam(\Omega)$, and an initial condition $\mu_0$ such that the solution starting from $\mu_0$ does not aggregate to a point.

Let $\Omega=\{(r\cos \theta, r\sin \theta)\in \mathbb{R}^{2}: 1-\epsilon \leq r \leq 1, -\epsilon\leq \theta\leq  \pi+\epsilon\}$ 
for $0<\epsilon<\frac{1}{2}$ be as shown in Figure \ref{fig:aggregation}. 
Let
$x^1=\left(-\left(1-\epsilon\right)\cos \epsilon, -\left(1-\epsilon\right)\sin \epsilon\right), 
x^2=\left(\left(1-\epsilon\right)\cos \epsilon, -\left(1-\epsilon\right)\sin \epsilon\right)$ and 
set $\mu_0=\frac{1}{2}\delta_{x^1}+\frac{1}{2}\delta_{x^2}$. 
Then $\Omega$ is $\eta$-prox-regular with 
$\eta= |x^1-x^2|/2 > 1-2\epsilon$. Since $\Diam(\Omega)=2$, thus $(\frac12 - 2 \veps) \Diam(\Omega) <\eta < \frac12 \Diam(\Omega)$. 
 For any radially symmetric $W$ which satisfies (A1) for some $\lambda_{W}>0$, 
a direct calculation yields that $v(0,x^1)=-\frac{1}{2}\nabla W(x^{1}-x^{2})\in N(\Omega, x^1)$. Thus $P_{x^1}\left(v(0,x^1)\right)=0$ and 
similarly $P_{x^2}\left(v(0,x^2)\right)=0$. We then see that $\mu(t)\equiv \mu_0$ is the solution to (\ref{conequ}), and hence aggregation to a singleton, (\ref{aggrecon2}), does not hold.
\begin{figure}[h]
\centering
\includegraphics[width=0.45\linewidth,angle=270]{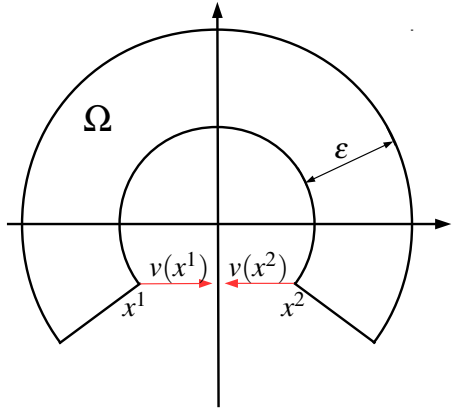}
\put(-170,-60){\Large $\Omega$}
\put(-145,-115){$v(x^1)$}
\put(-115,-115){$v(x^2)$}
\put(-75,-70){$\epsilon$}
\put(-95,-130){$x^2$}
\put(-155,-130){$x^1$}
\caption{The velocity $v$ at $x^1$ and $x^2$ are shown as the red arrows, which lie in the normal cones of the points respectively.}
\label{fig:aggregation}
\end{figure}
\end{remk}

\subsection*{Acknowledgements.}
The authors are grateful to Filippo Santambrogio for stimulating discussions and David Kinderlehrer for valuable suggestions and constant support.
DS is also grateful to NSF (grant DMS-1211760). 
LW acknowledges the support from NSF DMS 0806703.
The research was also supported by NSF PIRE grant  OISE-0967140. 
DS and LW are thankful to the Center for Nonlinear Analysis (NSF grant DMS-0635983) for its support.
JAC acknowledges support from projects MTM2011-27739-C04-02,
2009-SGR-345 from Ag\`encia de Gesti\'o d'Ajuts Universitaris i de Recerca-Generalitat 
de Catalunya, the Royal Society through a Wolfson Research Merit Award, and the Engineering 
and Physical Sciences Research Council (UK) grant number EP/K008404/1. The authors thank 
MSRI at Berkeley where part of this work was carried over during the program on Optimal Transport.
\\

\bibliography{non-convex-inter}{}
\bibliographystyle{siam}

\end{document}